\documentclass[11pt]{amsart}

   \renewcommand{\footnote}[1]{
\textsuperscript{
\addtocounter{footnote}{1}
(\thefootnote)
}
\footnotetext{#1}
}

\usepackage{url}
\usepackage{bm}
\usepackage{amsmath}
\usepackage{amsthm}

\usepackage{amssymb}
\usepackage{amsfonts}
\usepackage{graphicx}  
\usepackage[frenchb,english]{babel}
\usepackage[utf8]{inputenc}
\usepackage[T1]{fontenc}
\usepackage{mathrsfs}
\usepackage{enumerate}
\usepackage{version}
\usepackage{calc}
\usepackage{subfigure}

 \usepackage{pstricks,pstricks-add,pst-math,pst-xkey}
 \usepackage{float,graphics}
 \usepackage{indentfirst}
 \usepackage{yhmath} 
 
\theoremstyle{remark}

\newtheorem{defi}{Definition}

\newtheorem{remark}[defi]{Remark}

\theoremstyle{plain}
\newtheorem{thm}{Theorem}
\newtheorem{lem}[defi]{Lemma}

\newtheorem{prop}[defi]{Proposition}

\theoremstyle{Definition}

\def\v{\varepsilon}
\def\e{\varepsilon}

\def\O{\Omega}

\def\p{\partial}
\def\r{R}

\def\o{\omega}
\def\R{\mathbb{R}}
\def\C{\mathbb{C}}
\def\S{\mathbb{S}}

\def\deg{ {\rm deg}}
\def\n{\nabla}
\def\dist{{\rm dist}}
\def\tr{{\rm tr}}
\def\f{\varphi}
\def\weak{\rightharpoonup}

\def\N{\mathbb{N}}
\def\Z{\mathbb{Z}}
\def\J{\mathcal{J}}
\def\I{\mathcal{J}}

\def\1{\textrm{1\kern-0.25emI}}
\def\0{\textrm{0\kern-0.25emI}}

\def\dom{{\mathcal{D}}}
\def\K{\mathcal{K}}

\def\A{\mathcal{A}}
\def\Ar{\mathbb{A}}
\def\rmin{\rho_{\rm min}}

\DeclareMathOperator{\dive}{div}

\DeclareMathOperator{\capa}{cap}
\DeclareMathOperator{\tra}{tr}

\title[Existence and non existence results for G-L energy] {Existence and non existence results for minimizers of the Ginzburg-Landau energy with prescribed degrees }
\author{Mickaël {\sc Dos Santos} \and R\'emy {\sc Rodiac}}

\address{D\'epartement de Math\'ematiques, Universit\'e Paris-Est-Cr\'eteil, 61 avenue du G\'en\'eral de Gaulle, 94010 Cr\'eteil Cedex, France}
\email{mickael.dos-santos@u-pec.fr, \ remy.rodiac@u-pec.fr}

\date{}

\begin{document}

\begin{abstract}
Let $\dom =\O\setminus\overline{\o} \subset \R^2$ be a smooth annular type domain. We consider the simplified Ginzburg-Landau energy $E_\e(u)=\frac{1}{2}\int_\dom |\nabla u|^2 +\frac{1}{4\e^2}\int_\dom (1-|u|^2)^2$, where $u: \dom \rightarrow \C$, and look for minimizers of $E_\e$  with prescribed degrees $\deg(u,\partial \O)=p$, $\deg(u,\partial \o)=q$ on the boundaries of the domain. For large $\e$  and for balanced degrees ({\it i.e.}, $p=q$), we obtain existence of minimizers for {\it thin} domain. We also prove non-existence of minimizers of $E_\e$, for  large $\e$, in the case $p\neq q$, $pq>0$ and $\dom$ is a circular annulus with large capacity (corresponding to "thin" annulus). Our approach relies on similar results obtained for the Dirichlet energy $E_\infty(u)=\frac{1}{2}\int_\dom|\nabla u|^2$, the existence result obtained by Berlyand and Golovaty  and on a technique developed by Misiats.
\end{abstract}

\maketitle
\section{Introduction and main results}

We fix $\dom=\O\setminus\overline{\o}\subset\R^2$ a smooth annular type domain: $\O$ and $\o$ are smooth and 
bounded simply connected open sets s.t.  $\overline{\o}\subset\O\subset\R^2$. In this article, some results 
are specific to the case where $\dom$ is a circular annulus. In order to underline this specificity, when needed, we use the notation $\Ar=B(0,1)\setminus\overline{B(0,\r)}$ (with $\r\in]0,1[$) instead of $\dom$. 

\noindent We are interested in the existence or the non-existence of { global} minimizers of the Ginzburg-Landau type energy
\[
E_\v(u)=\frac{1}{2}\int_\dom|\n u|^2+\dfrac{1}{2\v^2}(1-|u|^2)^2
\] 
in the topological sectors of $\J:=\{u\in H^1(\dom,\C)\,|\,\tr_{\p\dom}(u)\in H^{1/2}(\p\dom,\S^1)\}$ for large values of $\v>1$. Here, $\tr_{\p\dom}$ stands for the {\it trace operator} on $\p\dom$ and $\S^1=\{x\in\C\,|\,|x|=1\}$.
We consider also the Dirichlet energy
\[
E_\infty(u)=\frac{1}{2}\int_\dom|\n u|^2,\,u\in\J.
\]

\noindent For $\Gamma\in\{\p\O,\p\o\}$ and for $u\in  \J$ we let
\[
\deg_\Gamma (u)=\frac{1}{2\pi}\int_\Gamma{u\wedge\p_\tau u\,{\rm d}\tau}.
\]
Here:
\begin{enumerate}[$\bullet$]
\item Each Jordan curve $\Gamma$ is directly (counterclockwise) oriented.
\item We let $\nu$ be the outward normal to $\O$ if $\Gamma=\p\O$ or $\o$ if $\Gamma=\p\o$, and $\tau=\nu^\bot$ is the tangential vector of $\Gamma$.
\item The differential operator $\p_\tau=\tau\cdot\n$ is the tangential derivative  and $"\cdot"$ stands for the usual scalar product in $\R^2$. 
We use also the standard notation $"\p_\nu"$ for the normal derivative $\p_\nu=\nu\cdot\n$.
\item The vectorial operator $"\wedge"$ stands for the vectorial product in $\C$, it is defined by $(z_1+\imath z_2)\wedge(w_1+\imath w_2):=z_1w_2-z_2w_1,\:z_1,z_2,w_1,w_2\in\R$.  
\item It is well known that $\deg_\Gamma (u)$ is an integer see \cite{BeMi1} (the introduction) or \cite{B1}.
\item The integral over $\Gamma$ should be understood using the duality between $H^{1/2}(\Gamma)$ and $H^{-1/2}(\Gamma)$ (see, \emph{e.g.}, \cite{BeMi1} Definition 1).
\item For $u\in\J$, we write $\deg(u)=(\deg_{\p\O} (u),\deg_{\p\o} (u))$.
\end{enumerate}

\noindent For $P=( p,q)\in\Z^2$, we are interested in the minimization of $E_\v$ for large $\v>1$ in 
\[
\J_{P}=\J_{p,q}:=\left\{u\in\J\,|\,\deg(u)=({p},q)\right\}.
\] 
For $\v\in]0,\infty]$ and $P=(p,q)\in\Z^2$, we denote
\[
m_\v(P)=m_\e(p,q)=\inf_{\J_P}E_\v.
\]
It is well known that the $\J_P$'s are the connected component of $\J$. They are open and closed for the strong topology induced by the $H^1$-norm. Hence if a minimizer of $E_\e$ in $\I_{p,q}$ exists for some $(p,q)\in \mathbb{Z}^2$ it satisfies the following Euler-Lagrange equations:

\begin{equation}\label{semi-stiff 1}
\left\{
\begin{array}{rclll}
-\Delta u & =&\dfrac{1}{\e^2}u(1-|u|^2) & \text{in} \ A \\
|u|&=&1 & \text{on }\partial  A \\
u\wedge \partial_\nu u & =& 0 &  \text{on} \ \partial A
\end{array}.
\right.
\end{equation}

These equations are obtained by making variations of the form $u_t=u+t\varphi$ for $t\in\R,\varphi \in C^\infty_0(\dom,\R^2)$ and $u_t=ue^{it\psi}$ for $t\in\R,\psi \in C^\infty(\overline{\dom},\R)$ (see Appendix C in \cite{BeMi0}). 

However the sets $\I_P$ are not closed with respect to the weak convergence in $H^1$ (see Introduction in \cite{BeMi0}). 
This fact implies that, in general, the minimization problem $m_\v(P)$ is not easy to handle since the direct minimization method fails. Namely in some cases $m_\v(P)$ is not attained. In contrast, for some other configurations, all minimizing sequence converges in $H^1$-norm. Such questions are central in this article.
\begin{remark}
It is obvious that for $p=q=0$ and $\v\in]0,\infty]$, $m_\v(0,0)$ is attained and the minimizers are the constants of modulus $1$. Thus we can focus on the case $(p,q)\neq(0,0)$.
\end{remark}

In this article we obtained existence and non existence results for {\it thin} domains.

\begin{defi}\label{DefConfRatio}We fix a conformal mapping 
\[
\Phi:\Ar=\{x\in\R^2\,|\,\r_\dom<|x|<1\}\to\dom.
\]
\begin{itemize}\item The number $\r_\dom\in]0,1[$ corresponds to the {\it conformal ratio} of $\dom$. 
\item When $\r_\dom$ is "close to" $1$, the domain $\dom$ is {\it thin}. When $\r_\dom$ is "close to" $0$, the domain $\dom$ is {\it thick}.
\item In this context the well known {\it $H^1$-capacity of $\dom$} is  $\capa(\dom)=-\dfrac{2\pi}{\ln\r_\dom}$.
\end{itemize}
\end{defi}

This article essentially contains two theorems. The first one is an existence result and, roughly speaking, states that for all $p \in \mathbb{N}^*$, under an hypothesis \eqref{ExistKeyHyp} (which expresses  that the annulus is  thin) and if $\e$ is sufficiently large then $m_\e(p,p)$ is attained.

\begin{thm}\label{PropEXist1}
Let $\dom\subset\R^2$ be an annular type domain and let $p\in \mathbb{N}^*$. If
\begin{equation}\tag{H}\label{ExistKeyHyp}
m_\infty(p,p)< m_\infty(p-1,p-1)+2\pi
\end{equation} 
then there exists $\e_p>0$ s.t. if $\e_p< \e \leq +\infty$ then minimizing sequences for $m_\e(p,p)$ are compact (for the $H^1$-norm). In particular $m_\e(p,p)$  is attained. 

For $(u_\v)_{\v>\v_p}\subset\J_{p,p}$ a sequence of minimizer there is $u_\infty\in\J_{p,p}$ a minimizer for $m_\infty(p,p)$ s.t., up to a subsequence, we have: 
$$u_\e \underset{\v\to\infty}{\rightarrow} u_\infty \ \text{in} \ C^{l}(\overline{\dom})\:\forall l\in\N. $$
\end{thm}
\begin{remark}
\begin{enumerate}
\item Since $\J_{-p,-p}=\{\overline{u}\,|\,u\in\J_{p,p}\}$ where $\overline{u}$ is the conjugate of $u$ and since $E_\v(\overline{u})=E_\v(u)$, it is easy to reformulate Theorem \ref{PropEXist1} for $p<0$.
\item The condition \eqref{ExistKeyHyp} is theoretical. We are able to prove that this condition holds true under the following condition of capacity of the domain. There exists $0<R_p<1$ s.t. if the conformal \textit{ratio}  $R_\dom$ satisfies $R_p<R_\dom<1$ then $\eqref{ExistKeyHyp}$ holds. Note that $\r_p$ is the same than in Theorem \ref{P-GolBer} below.

\item Note that for $1>\r_\dom>\r_p$ we have that the minimizers of $m_\infty(p,p)$ are vortexless. Consequently, for sufficiently large $\v$, the minimizers of $m_\v(p,p)$ are also vortexless. 
\end{enumerate}
\end{remark}

The previous theorem is an "extension" to general annular type domains of a previous result of Berlyand and Golovaty:
\begin{thm}[\cite{GB1}]\label{P-GolBer}
Let $p \in \mathbb{N}^*$ there exists a critical outer radius $0< \r_p<1$ s.t. for $\r_p<\r<1$, $m_\e(p,p)$ is attained by a unique (up to a phase) radially symmetric minimizer for all $0<\e < +\infty$.
\end{thm}
\begin{defi}
In the previous theorem, the expression "{\it up to a phase}" means that if $u$ is a minimizer, then $\tilde u$ is a minimizer if and only if there exists $\alpha\in\S^1$ s.t. $\tilde u=\alpha u$. 
Another way to explain this expression is to say that two minimizers have pointwise same moduli and the difference of their phases is a constant.
\end{defi}

\begin{remark}
Theorem \ref{P-GolBer} may be easily extended to the case $\v=\infty$. {[see Step 2 in the proof of Proposition \ref{CompaPropBG}]}
\end{remark}
Although  Theorem \ref{PropEXist1} may be seen as an extension of Theorem \ref{P-GolBer}, the methods used in their proofs  are different. Condition  \eqref{ExistKeyHyp} allows to make arguments in the spirit of concentration-compactness phenomenon and bubbling analysis (see \textit{e.g.} \cite{Brezislackofcompactness}). See Section \ref{SectAnalCriti} for  a detailed comparison between both theorems. \\
Note that in \cite{FarMir} (Theorem 1.5), Farina and Mironescu have also extended Theorem \ref{P-GolBer}, to general annular type domains. They proved that there is some explicit universal constant $\delta\simeq 0.045$ such that if $m_\e(p,p)<\delta$ then the infimum is attained and the minimizer is unique (up to a phase). Then using $\mathbb{S}^1$-valued test functions, and the conformal invariance of the Dirichlet energy, they obtained that if the annular domain is very thin then the condition $m_\e(p,p)<\delta$ holds. Their condition on the thinness of the annular domain is more restrictive than ours, however they obtained a more precise result: uniqueness of minimizer (up to a phase). We want to emphasize that the proof of uniqueness  is a real challenge (existence is direct for $\delta<\pi$).\\

Our second theorem is a non-existence result specific to the symmetric case $\dom= \Ar =B(0,1)\setminus\overline{B(0,\r)}$ with $\r$ close to $1$.

\begin{thm}\label{ThmNonExist}
Let $p,q\in\N^*$ s.t.  $p\neq q$. Then there are $0<\r_{\min(p,q)}<1$ and $\v_{\min(p,q)}>1$ s.t. for $ \r_{\min(p,q)}<\r<1$, $\Ar =B(0,1)\setminus\overline{B(0,\r)}$ and $\v>\v_{\min(p,q)}$ we have $m_\v(p,q)$ is not attained.
\end{thm} 
A technique to prove non existence of minimizers [or local minimizers] with prescribed degrees for the Ginzburg-Landau energy was devised by Berlyand, Golovaty and Rybalko in \cite{BGR1}. They proved the non existence of minimizers of $E_\e$ in $\J_{1,1}$ for thick annular domain. Then, perfecting this technique, Misiats proved the non existence of  minimizers in some subset of $\I_{p,q}$ in \cite{Misiats1}. The first non existence result for global minimizers of the  Ginzburg-Landau energy with prescribed degrees $p\neq q$ and $pq>0$ was obtained by Mironescu in \cite{Size} following the ideas of Berlyand, Golovaty, Rybalko and Misiats. It can be rephrased as follows:

\begin{thm}(Thm 4.16-\cite{Size})\label{nonexistMir}
Let $p,q\in \mathbb{N}^*$, $pq>0$ then there exists a critical value of the capacity $C_{\min(p,q)}>0$ s.t. if $\capa(\dom) < C_{\min(p,q)}$ then $m_\e(p,q)$ is not attained for $\e$ small.
\end{thm}

\begin{remark}
Note that in the previous theorem the annulus is "thick", {\it i.e.}, $\capa(\dom)$ is small and $\e$ is small. Hence we are in the opposite situation of Theorem \ref{ThmNonExist}. However the proofs of these two theorems follow the same ideas. Note also that we can have $p=q$ in Theorem \ref{nonexistMir}. 
\end{remark}

Our approach is similar to the one mentioned before. In particular we follow the strategy of Misiats in \cite{Misiats1}. The new ingredient which allows us to obtain Theorem \ref{ThmNonExist} is a non existence result for minimizers of $E_\infty$ in $\J_{p,q}$ with $pq>0$ obtained in \cite{HR}  using the so-called Hopf quadratic differential. 

Before doing the proofs of both theorems (see Sections \ref{ProofThmExist}\&\ref{ProofThmNonExist}) we recall some classical results:

\begin{itemize}
\item[$\bullet$]In Section \ref{SectSomeTechRes} we recall some basic results used to prove Theorems \ref{PropEXist1}\&\ref{ThmNonExist}. 
\item[$\bullet$] In Sections \ref{LiteratureEpsNeqInfty}\&\ref{LiteratureEps=Infty} we list some results about the existence or the non existence of solution for $m_\v(p,q)$ for $\v\in]0,\infty[$ (Section \ref{LiteratureEpsNeqInfty}) or $\v=\infty$ (Section \ref{LiteratureEps=Infty}).
\end{itemize}
 \section{Some "basic" results and some pieces of the literature}
\subsection{Bound for $m_\v(p,q)$ and cost to move  degrees}\label{SectSomeTechRes}
In the following for $(p,q),(p',q')\in\Z^2$, we denote 
\[
|(p,q)|=|p|+|q|\text{ and }|(p,q)-(p',q')|=|p-p'|+|q-q'|.
\]
\begin{prop}\label{Property1Meps}
Let $P,P'\in\Z^2$. For $0<\v'<\v\leq\infty$ we have:
\begin{enumerate}
\item $m_{\v}(P)\leq\pi|P|$,
\item $m_{\v}(P)\leq m_\v(P')+|P-P'|$,
\item $|m_{\v}(P)-m_{\v'}(P)|\to0$ if $\v'\uparrow\v$.
\end{enumerate}
\end{prop}
\begin{remark}
Note that in the third assertion we may replace $\v'\uparrow\v$ by $\v'\to\v$ but in the following we only need $\v'\uparrow\v$.
\end{remark}
\begin{proof}
The two first assertions of Proposition \ref{Property1Meps} are direct consequences of Proposition \ref{PropStandBub} below. 

We prove the third assertion. For $P\in\Z^2$ and $\v'\uparrow\v\in]0,\infty]$ we consider $(u_{\v'})_{\v'}$ a minimizing sequence of $m_\v(P)$ s.t.  $-\Delta u_{\v'}=\dfrac{u_{\v'}}{\v^2}(1-|u_{\v'}|^2)$. It is clear that such minimizing sequence always exists. Thanks to the maximum principle (see \textit{e.g.} Proposition 2 \cite{BBH0}), we have $|u_{\v'}|\leq1$. Since $\v'<\v$ we have 
\begin{eqnarray*}
E_{\v'}(u_{\v'})\geq m_{\v'}(P)\geq m_\v(P)=E_\v(u_{\v'})-o_{\v'}(1)
\end{eqnarray*}
where $o_{\v'}(1)\to0$ when $\v'\to\v$.

\noindent We denote
\[
\K({\v'})=\begin{cases}\frac{1}{4\v'^2}-\frac{1}{4\v^2}&\text{if }\v\neq\infty\\\frac{1}{4\v'^2}&\text{if }\v=\infty\end{cases}.
\]
It is clear that we have $\K({\v'})\to0$ when $\v'\to\v$. Therefore we have
\begin{eqnarray*}
\K({\v'})|\dom|&\geq&\K({\v'})\int_\dom(1-|u_{\v'}|^2)^2
\\&=&E_{\v'}(u_{\v'})-E_{\v}(u_{\v'})\geq m_{\v'}(P)-m_\v(P)+o_{\v'}(1).
\end{eqnarray*}
Here $|\dom|$ is the measure of $\dom$. Since $m_{\v'}(P)-m_\v(P)\geq0$ we thus obtain that $m_{\v'}(P)-m_\v(P)\to0$ when $\v'\uparrow\v$.
\end{proof}

\begin{prop}\label{PropStandBub}[Standard bubbling]
Let $\v\in]0,\infty]$, $\eta>0$, ${\bf e}\in\{(1,0),(0,1)\}$ and $u\in\J$. There are $v_+,v_-\in\J$ s.t. $v_+\in\J_{\deg(u)+{\bf e}},\,v_-\in\J_{\deg(u)-{\bf e}}$ and
\begin{eqnarray}\label{SharpBubble}
E_\v(v_+)\leq E_\v(u)+\pi+\eta, \\
E_\v(v_-)\leq E_\v(u)+\pi+\eta. \label{SB2}
\end{eqnarray}
\end{prop}
The proof of Proposition \ref{PropStandBub} may be found in \cite{Dos1} Lemma 7.

In order to drop $\eta$ in \eqref{SharpBubble} and  \eqref{SB2} and to replace the large inequality by a strict inequality, we need an extra-hypothesis about the behavior of $u$ on the connected component of $\p\dom$ where the degree is modified. 
\begin{prop}\label{ShaprBBL}Let $\v\in]0,\infty]$ and let $u\in\J_{p,q}$ be any function which satisfies $|u|\leq1$ in $\dom$ and $\p_\nu|u|>0,\,u\wedge\p_\nu u=0$ on $\p\O$.
\begin{enumerate}
\item Assume that there is $x_0\in\p\O$ s.t. $u\wedge\p_\tau u(x_0)>-u\cdot\p_\nu u(x_0)$ then there exists $v\in\J_{p-1,q}$ s.t. $E_\v(v)<E_\v(u)+\pi$.
\item Assume that there is $x_0\in\p\O$ s.t. $u\wedge\p_\tau u(x_0)<u\cdot\p_\nu u(x_0)$ then there exists $v\in\J_{p+1,q}$ s.t. $E_\v(v)<E_\v(u)+\pi$.
\end{enumerate}
An analogous lemma can be stated considering the other boundary $\p\o$.
\end{prop}
Proposition \ref{ShaprBBL} is proved in \cite{RS1} (Lemma 1.2).

One of the main tool in the study of the minimization of $E_\v$ in $\J_{p,q}$ is the beautiful {\it Price Lemma}. As explain before, the degree $\deg:\J\to\Z^2$ is not continuous for the weak $H^1$ convergence, this lemma expresses the energetic cost to modify degrees  for a weak $H^1$-limit.
\begin{lem}[Price Lemma see Lemma 1 in {\cite{BeMi1}}] \label{PriceLemma}
Let $P\in\Z^2$ and $(u_n)_n\subset\J_P$ s.t. $u_n\weak u$ in $H^1(\dom)$. Then
\[
\liminf_{n \rightarrow +\infty} E_\infty(u_n)\geq E_\infty(u)+\pi|P-\deg(u)|.
\]
Using Sobolev embeddings it also holds that, for all $\e>0$:
\[
\liminf_{n \rightarrow +\infty} E_\e(u_n)\geq E_\e(u)+\pi|P-\deg(u)|.
\]
\end{lem}
\subsection{Some known Existence/Non Existence results: the case $\v\in]0,\infty[$}\label{LiteratureEpsNeqInfty}
The first non existence result is certainly the following.
\begin{prop}
Let $\v>0$, if $(p,q)\in\Z^2$ are s.t. $(p,q)\neq(0,0)$ and $pq\leq0$, then $m_\v(p,q)$ is not attained.
\end{prop}
\begin{proof}
The starting point of the proof are  the two following estimates :
\begin{itemize}
\item the pointwise inequality $|\n u|^2\geq 2|{\rm Jac\,} u|$ [here ${\rm Jac}\, u=u_x \wedge u_y$ is the Jacobian of $u$];
\item the degree formula valid for $u\in\J$ (see \textit{e.g.} (1.6) in \cite{Oldandnew}) :
\begin{equation}\label{oldformula}
\left|\int_\dom{\rm Jac}\, u\right|=\pi|\deg_{\p\O}(u)-\deg_{\p\o}(u)|.
\end{equation}
\end{itemize}
By combining both previous estimates, if $pq\leq0$, then for all $u\in\J_{p,q}$, we easily obtain that 
\[
\dfrac{1}{2}\int_\dom|\n u|^2\geq\pi(|p|+|q|).
\]
On the other hand, by Proposition \ref{Property1Meps}.1 it holds that
\[
\inf_{\J_{p,q}}E_\v\leq \pi(|p|+|q|).
\]
By combining both bounds, we obtain 
\[
\inf_{\J_{p,q}}E_\v= \pi(|p|+|q|).
\]

\noindent Now we argue by contradiction and we assume that there exists $\v>0$ s.t. $m_\v(p,q)$ is attained by $u_\v$. Then we have 
\[
\pi(|p|+|q|)=\dfrac{1}{2}\int_\dom|\n u_\v|^2=E_\v(u_\v).
\]
Therefore $\int_\dom(1-|u_\v|^2)^2=0$, {\it i.e.}, $u_\v\in H^1(\dom,\S^1)$. Since $u_\e$ is $\mathbb{S}^1$-valued we have $\text{ Jac} \ u_\e =0$ and the degree formula \eqref{oldformula} implies that $p=q$. This fact is in contradiction with $(p,q)\neq(0,0)$ and $pq\leq0$.
\end{proof}

 Our main results deal with the remaining cases: $pq>0$. It is obvious that this condition means $p,q>0$ or $p,q<0$. Without lack of generality we may assume that $p,q>0$ (since $\deg(\overline{u},\Gamma)=-\deg(u,\Gamma)$ for  $\Gamma \in \{\partial \Omega, \partial \omega) \})$.\\

In an annular $\Ar=B(0,1)\setminus\overline{B(0,\r)}$, a natural candidate to be a minimizer for $m_\v(p,p)$ is the {\it radial Ginzburg-Landau solution of degree $p$}. The radial Ginzburg-Landau solution  of degree $p$  is a special solution of the semi-stiff problem
\begin{equation}\nonumber
\begin{cases}
-\Delta u=\dfrac{u}{\v^2}(1-|u|^2)^2&\text{in }\Ar\\
|u|=1,\,u\wedge\p_\nu u=0&\text{on }\p\Ar
\end{cases}.
\end{equation}
This solution is of the form 
\begin{equation}\label{EqRadEpsSol}
u_{\v,p}(x)=\rho_{\v,p}(|x|)\left(\dfrac{x}{|x|}\right)^p
\end{equation}
 where $\rho_{\v,p}\in C^\infty([\r,1],[0,1])$ is the unique solution of 
\begin{equation}\label{ModulusEqEpsRadSol}
\begin{cases}
-\rho''-\dfrac{\rho'}{r}+\dfrac{p^2 \rho}{r^2}=\dfrac{\rho}{\v^2}(1-\rho^2)\text{ in }]\r,1[\\\rho(\r)=\rho(1)=1
\end{cases}.
\end{equation}

As seen in the introduction, Berlyand and Golovaty proved a very precise existence result (see Theorem 2.13 in \cite{GB1}) for the minimization of $E_\v$ in $\J_{p,p}$ with $p\geq1$ in annulars $\Ar=B(0,1)\setminus\overline{B(0,\r)}$ for $\r$ sufficiently close to $1$.

For the special cases $p=q=1$ and for an annular type domain $\dom$, by using a compilation of works of Berlyand, Golovaty, Mironescu and Rybalko (see {\it e.g.} \cite{BeMi0}, \cite{BeMi1}, \cite{BGR1}) we may state the following proposition:
\begin{prop}\label{Prop-Compile}
Let $\dom\subset\R^2$ be an annular type domain and let $\r_\dom$ be the conformal {\it ratio} of $\dom$.
\begin{itemize}
\item[$\bullet$] If $\r_\dom\leq {\rm e}^2$ then $m_\v(1,1)$ is attained for all $\v$.  
\item[$\bullet$] If $\r_\dom> {\rm e}^2$ then then there is $\v_0>0$ s.t., for $\v>\v_0$, $m_\v(1,1)$ is attained and, for $\v<\v_0$, $m_\v(1,1)$ is not attained.
\end{itemize}
\end{prop}
\subsection{Some Existence/Non Existence results: the case $\v=\infty$}\label{LiteratureEps=Infty}
In the case of the Dirichlet energy, thanks to the conformal invariance of $E_\infty$, we may restrict the study to a ring $\Ar=B(0,1)\setminus\overline{B(0,\r)}$ with $\r\in]0,1[$.

As for the study of the minimization of the Ginzburg-Landau energy in a ring, a natural candidate to minimize the Dirichlet energy in $\J_{p,p}$ is the {\it radial harmonic map of degree $p$} which solves the semi-stiff problem
\begin{equation}\nonumber
\begin{cases}
\Delta u=0&\text{in }\Ar\\
|u|=1,\,u\wedge\p_\nu u=0&\text{on }\p\Ar
\end{cases}.
\end{equation}
This solution is of the form 
\begin{equation}\label{RadHarmSol}
u_{\infty,p}(x)=\rho_{\infty,p}(|x|)\left(\dfrac{x}{|x|}\right)^p
\end{equation}
where $\rho_{\infty,p}\in C^\infty([\r,1],[0,1])$ is the unique solution of 
\begin{equation}\label{ModulusEqRadSol}
\begin{cases}
-\rho''-\dfrac{\rho'}{r}+\dfrac{p^2 \rho}{r^2}=0\text{ in }]\r,1[\\\rho(\r)=\rho(1)=1
\end{cases}.
\end{equation}
In an unpublished paper, Berlyand and Mironescu [Lemma D.3 in \cite{BeMi0}] proved the following proposition that treats the case $p=q=1$.
\begin{prop}\label{ExistCaseDeg1Epsinfty}
For all $\r\in]0,1[$, the radial harmonic map of degree $1$ is the unique [up to a phase] minimizer of $m_\infty(1,1)$.
\end{prop}
Next, Hauswirth and Rodiac in \cite{HR} considered the problem $m_\infty(p,q)$ for $p,q\in\Z$. They proved the following proposition:
\begin{prop}\label{PropEpsInftHR}
Let $p,q\in\Z$ then we have
\begin{itemize}
\item[$\bullet$] If $p\neq q$  and $pq>0$ then $m_\infty(p,q)$ is not attained. Without loss of generality we can assume that $p>q>0$ and then it holds that $m(p,q)=m(q,q)+2\pi(p-q)$.
\item[$\bullet$] If $p=q\neq0$ then there is $0<\r_{p}<1$ s.t. for $\r_p<\r<1$  $m_\infty(p,p)$  is attained and the radial harmonic map of degree $p$ is the unique [up to a phase] minimizer of $m_\infty(p,p)$.
\end{itemize}
\end{prop}

\begin{remark}Note that the radius $\r_p$ obtained by Hauswirth and Rodiac is the same as the radius obtained by Berlyand and Golovaty (see Theorem \ref{P-GolBer}) and that if $p>p'$ then $\r_p\geq\r_{p'}$ (see Step 1 in the proof of Proposition \ref{CompaPropBG}).
\end{remark}
\section{Existence Result}\label{ProofThmExist}
This section is dedicated to the proof of Theorem \ref{PropEXist1}. We first study the behavior as $\e_n$ goes to some $\e_* \in ]0,+\infty]$ of sequences $(u_n)$ s.t. $u_n$ is almost minimizing for $E_{\e_n}$. Then we derive a theoretical condition [Hyp. \eqref{ExistKeyHyp}] under which the compactness of minimizing sequences for $E_\e$ holds for large $\e$. At last we compare Hyp. \eqref{ExistKeyHyp} with the condition of Theorem \ref{P-GolBer}. 

\subsection{The key argument}
For $(p,q)\in\Z^2$ we define 
\[
\A_{(p,q)}=\left\{(p',q')\in\Z^2\,|\,pp'\geq0,\,|p'|\leq|p|\text{ and }qq'\geq0,\,|q'|\leq|q|\right\}.
\]
\begin{lem}\label{LFond1}
Let $P=(p,q)\in\Z^2$, $\v_*\in]0,\infty]$ and $(\v_n)_n$ be an increasing sequence s.t. $\v_n\uparrow\v_*$ or $\v_n=\v_*$ for all $n$. Consider a sequence $(u_n)_n\subset\J_P$ s.t.
\[
E_{\v_n}(u_n)\leq m_{\v_n}(P)+o_n(1).
\]
By Proposition \ref{Property1Meps}.1,  there is $u\in\J_{P'}$ s.t., up to a subsequence, $u_n\weak u$. Then $P'\in\A_P$ and $u$ minimizes $m_{\v_*}(P')$.
Moreover, if $P'\neq P$ then $m_{\v_*}(P)=m_{\v_*}(P')+\pi|P-P'|$.
\end{lem}
\begin{proof}
Fix $P=(p,q)\in\Z^2$, $\v_*\in]0,\infty]$, $(\v_n)_n$,  be an increasing sequence s.t. $\v_n\uparrow\v_*$  or $\v_n=\v_*$ for all $n$ and a sequence $(u_n)_n\subset\J_P$ s.t.
\[
E_{\v_n}(u_n)\leq m_{\v_n}(P)+o_n(1).
\]
There exists $u \in\J_{P'}$ s.t., up to a subsequence, $u_n\weak u$. By the Price Lemma (Lemma \ref{PriceLemma}) we have
\[
\liminf_nE_\infty(u_n)\geq E_\infty(u)+\pi|P-P'|.
\]
On the other hand, up to pass to an extraction we have $|u_n|\to|u_\infty|$ in $L^4$ we thus have:
\[
\frac{1}{4\v_n^2}\int_\dom(1-|u_n|^2)^2\underset{n\to\infty}{\longrightarrow}\left|\begin{array}{cl}\displaystyle\frac{1}{4\v_*^2}\int_\dom(1-|u_\infty|^2)^2&\text{if }\v_*<\infty\\0&\text{if }\v_*=\infty\end{array}\right..
\]
By combining the two previous estimates we obtain:
\[
\liminf_nE_{\v_n}(u_n)\geq E_{\v_*}(u)+\pi|P-P'|.
\]
From Proposition \ref{Property1Meps}.2\&3 we deduce:
\begin{eqnarray}\nonumber
m_{\v_*}(P')+\pi|P-P'|&=&\lim_n m_{\v_n}(P')+\pi|P-P'|
\\\nonumber&\geq&\lim_n m_{\v_n}(P)
\\\nonumber&=&\liminf_nE_{\v_n}(u_n)
\\\label{ContrdictLem1}&\geq&E_{\v_*}(u)+\pi|P-P'|.
\end{eqnarray}
Therefore we have  $u\in\J_{P'}$ and $m_{\v_*}(P')\geq E_{\v_*}(u)$. Consequently $u$ minimizes $m_{\v_*}(P')$.

\noindent Assume now that $p\geq 0$ and that $p'>p$. 
 
\noindent Note that $u$ satisfies the hypotheses of Proposition \ref{ShaprBBL} and that there exists $x_0\in\p\O$ s.t. $u\wedge\p_\tau u(x_0)>0$ because $\deg_{\p\O}(u)>0$ and $-u(x_0)\cdot\p_\nu u_\infty(x_0)=-\dfrac{1}{2}\p_\nu|u_\infty|^2(x_0)\leq0$ because $x_0$ is a maximum point of $|u_\infty|^2$ (recall that $|u|=1$ on $\partial \dom$ and $|u| \leq 1$ in $\dom$ thanks to the maximum principle).
 
\noindent By  Propositions \ref{PropStandBub}$\&$\ref{ShaprBBL} we have the existence of $\tilde{u}\in\J_P$ s.t. 
\begin{eqnarray}\nonumber
m_{\v_*}(P)&\leq& E_{\v_*}(\tilde{u})\\\nonumber&<&E_{\v_*}(u)+\pi|P-P'|\\\label{EqRajoutNum}&=&m_{\v_*}(P')+\pi|P-P'|.
\end{eqnarray}
By mimicking the argument which gives \eqref{ContrdictLem1} we obtain
\begin{eqnarray}\nonumber
m_{\v_*}(P)&=&\lim_{n} m_{\v_n}(P)
\\\nonumber&\geq&\liminf_nE_{\v_n}(u_n)
\\\label{ContrdictLem2}&\geq&m_{\v_*}(P')+\pi|P-P'|.
\end{eqnarray}
Clearly \eqref{ContrdictLem2} is in contradiction with \eqref{EqRajoutNum}. Thus if $p\geq 0$ then $p'\leq p$. Using the same argument we prove that if $p\geq 0$ then $p'\geq0$ and therefore $p'\in[0,p]$. If $p\leq0$, we obtain, through the same method, that $p'\in[p,0]$. The same results hold for $q$ instead of $p$.
\noindent Hence we obtain that $P'\in\A_P$.

\noindent We now prove the last part of the proposition. Noticing that the inequalities which give \eqref{ContrdictLem1} are in fact equalities, with the help of Proposition \ref{Property1Meps}.3  we deduce that $m_{\v_*}(P)=m_{\v_*}(P')+\pi|P-P'|$.

\end{proof}

\subsection{Consequences of the key argument : existence of minimizers}

The key argument describes what can happen to  almost minimizing sequences $(u_n)_n$ for $m_{\e_n}(p,q)$ when $\e_n$ tends to $\e_*$. Roughly speaking, if $p,q >0$, $u_n$ converges weakly to some $u$ in $H^1$. We have that $u \in \I_{r,s}$ with $0\leq r \leq p$, $0\leq s \leq q$, $u$ minimizes $E_{\e_*}$ in $\I_{r,s}$ and the loss of energy is quantified that is $m_{\e_*}(r,s)=m_{\e_*}(p,q)-\pi(p-r+q-s)$. We can then show that a sharp inequality [Hyp. \eqref{ExistKeyHyp}] prevents minimizing sequences from falling in a class $\I_{r,s}$ with $r\neq p$ and $s \neq p$. 

\begin{prop}\label{P-ExistPartThm1}
Let $\dom \subset \R^2$ be an annular type domain and let $p \in \mathbb{N}^*$ s.t.
\begin{equation}\tag{\ref{ExistKeyHyp}}
m_\infty(p,p)<m_\infty(p-1,p-1)+2\pi.
\end{equation}
Then, for sufficiently large $\e$, the minimizing sequences for $m_\e(p,p)$ are compact in $H^1(\dom)$ and thus $m_\e(p,p)$ is attained.
\end{prop}
\begin{proof}

We argue by contradiction. We assume that 
\begin{itemize}
\item $p\in\N^*$ and $\dom$ are s.t. \eqref{ExistKeyHyp} holds,
\item there exists $\v=\v_k\uparrow\infty$ s.t. for all $\v$ there is a  minimizing sequence $(u_n^\v)_n$ for $m_\v(p,p)$ satisfying:
\[
\text{$(u_n^\v)_n$ is not compact for the strong topology of $H^1$}.
\] 
\end{itemize}
\noindent For all $\v=\v_k$, up to consider an extraction in $(u_n^\v)_n$,  there is $u_\v\in\J$ s.t. $u_n^\v\underset{n\to\infty}{\weak} u_\v$ in $H^1(\dom)$. By Lemma \ref{LFond1}, we have that $\deg(u_\v)\in\A_{(p,p)}$ and that $u_\v$ minimizes $m_\v(\deg(u_\v))$.
 
\noindent Note that the minimizing property of $(u_n^\v)_n$ combined with its non compactness property,  imply that 
\begin{equation}\label{ContradThmNum1}
\deg(u_\v)\neq (p,p).
\end{equation}
Indeed, if $\deg(u_\v)= (p,p)$, then $u_\v\in\J_{p,p}$.  Moreover, by compact Sobolev embedding we have  $\displaystyle\lim_n\dfrac{1}{4\v^2}\int_\dom(1-|u_n^\v|^2)^2=\dfrac{1}{4\v^2}\int_\dom(1-|u_\v|^2)^2$. On the other hand $\lim_n E_\v(u_n^\v)=m_\v(p,p)=E_\v(u_\v)$. 

\noindent Consequently  $\displaystyle\liminf_n\dfrac{1}{2}\int_\dom|\n u_n^\v|^2= \int_\dom|\n u_\v|^2$ which implies that $u_n^\v\to u_\v$ in $H^1(\dom)$. This convergence contradicts the non compactness property of $(u_n^\v)_n$.\\

It is clear that the set $\{\deg(u_\v)\}\subset\A_{(p,p)}$ is finite. Thus we may consider an extraction, still denoted by $(\v_k)_k$, s.t. $\deg(u_\v)=P_1\in\A_{(p,p)}\setminus\{(p,p)\}$.
Up to an extraction in $(\v_k)_k$, there exists $u_\infty\in\J$ s.t. $u_\v\weak u_\infty$. By Lemma \ref{LFond1} we have that $P_2:=\deg(u_\infty)\in\A_{P_1}\subset\A_{(p,p)}$ and $u_\infty$ minimizes $m_\infty(P_2)$. Therefore by Proposition \ref{PropEpsInftHR} there is $p_2\in[0,p]$ s.t. $P_2=(p_2,p_2)$. Moreover, since  $P_2=(p_2,p_2)\in\A_{P_1}\subset \A_{(p,p)}\setminus\{(p,p)\}$ we have $p_2\in[0,p-1]$. Hence it holds that (by Prop. \ref{Property1Meps}.2)

\begin{eqnarray*}
m_\infty(P_2)+\pi|(p-1,p-1)-P_2|+2\pi&{\geq}&m_\infty(p-1,p-1)+2\pi
\\\text{[Hyp. \eqref{ExistKeyHyp}]}&>&m_\infty(p,p)\\
\text{[Prop. \ref{Property1Meps}.3]}&=&\lim_{\v\to\infty}m_\v(p,p)
\\&=&\lim_{\v\to\infty}\liminf_n E_\v(u^n_\v)
\\\text{[Lemma \ref{PriceLemma}]}&\geq&\lim_{\v\to\infty} E_\v(u_\v)+\pi|(p,p)-P_1|
\\&\geq&m_\infty(P_2)+\pi|P_2-P_1|+\\&&\phantom{ghsghaaasgs}+\pi|P_1-(p,p)|.
\end{eqnarray*}

\noindent Then we deduce that:
\[
|(p-1,p-1)-P_2|+2>|P_2-P_1|+|P_1-(p,p)|.
\]
By the triangle inequality we have:
\[
|(p-1,p-1)-P_2|+2>|P_2-(p,p)|.
\]
 Since $P_2=(p_2,p_2)$ with $p_2\in[0,p-1]$, the last inequality means
 \[
 p-p_2 > p-p_2.
 \]
 This is clearly a contradiction and the proposition is proved.
 \end{proof}

By using the same strategy as in the proof of Proposition \ref{P-ExistPartThm1} we have:
\begin{prop}\label{compactinfty}
Let $p>0$ and $\dom$ an annular type domain s.t. 
\begin{equation}\tag{\ref{ExistKeyHyp}}
m_\infty(p,p)<m_\infty(p-1,p-1)+2\pi
\end{equation}
holds. Then minimizing sequences for $m_\infty(p,p)$ are compact in $H^1$ and thus $m_\infty(p,p)$ is attained.
\end{prop}
\begin{proof}Let $p>0$. Assume that $m_\infty(p,p)<m_\infty(p-1,p-1)+2\pi$. Consider $(u_n)_n$ a minimizing sequence for $m_\infty(p,p)$. Up to pass to a subsequence we have the existence of $u_\infty\in\J$ s.t. $u_n\weak u_\infty$. Let $P':=\deg(u_\infty)$. If $P'=(p,p)$ then we are done. \\

\noindent Otherwise we have: $P'\neq(p,p)$. By Lemma \ref{LFond1} we have that $u_\infty$ minimizes $m_\infty(P')$ and $P'\in\A_{(p,p)}$. Thus, by Proposition \ref{PropEpsInftHR} we have the existence of $p'\in[0,p-1]$ s.t. $P'=(p',p')$ {\it [here we used $P'\neq(p,p)$]}.

\noindent Using Lemma \ref{LFond1} again we have
\begin{eqnarray*}
m_\infty(p-1,p-1)+2\pi&\stackrel{\eqref{ExistKeyHyp}}{>}&m_\infty(p,p)
\\&=&m_\infty(P')+2\pi(p-p').
\end{eqnarray*}
\noindent Therefore we obtained $m_\infty(p-1,p-1)>m_\infty(p',p')+2\pi|p-p'-1|$. This estimate is in contradiction with Proposition \ref{Property1Meps}.2.

\noindent Consequently we have $P'=(p,p)$ and then $m_\infty(p,p)$ is attained.
\end{proof}

\subsection{Comparaison with the work of Berlyand$\&$Golovaty \cite{GB1}}\label{SectAnalCriti}

This section is essentially dedicated to the proof of the following proposition

\begin{prop}\label{CompaPropBG}
Let $p\in\N^*$ and let $0<\r_p<1$ of Theorem \ref{P-GolBer}. For a annular type domain $\dom$ s.t. its conformal {\it ratio} [see  Definition \ref{DefConfRatio}] satisfies $\r_p<\r_\dom<1$ we have $m_\infty(p,p)<m_\infty(p-1,p-1)+2\pi$.
\end{prop}

Proposition \ref{CompaPropBG} as two direct consequences :
\begin{enumerate}
\item If the hypothesis of Theorem \ref{P-GolBer} holds for an annular $\Ar$ then Proposition \ref{P-ExistPartThm1} holds.
\item A way to reformulate (in a weaker form) the hypothesis of Theorem \ref{PropEXist1} or Proposition \ref{P-ExistPartThm1} is to replace "$m_\infty(p,p)<m_\infty(p-1,p-1)+2\pi$" by : 
\begin{itemize}
\item[$\bullet$]the conformal {\it ratio} of $\dom$ satisfies $\r_p<\r_\dom<1$ ($0<\r_p<1$ of Theorem \ref{P-GolBer});
\item[]or equivalently  
\item[$\bullet$]$\capa(\dom)>C_p$ for $C_p=\dfrac{-2\pi}{\ln \r_p}$.
\end{itemize}
\end{enumerate}
\begin{proof}
We prove Proposition \ref{CompaPropBG} in 3 steps.\\

\noindent {\bf Step 1.} The sequence of critical radii $(\r_p)_{p\geq1}$ of Theorem \ref{P-GolBer} is non decreasing\\

The critical radius $\r_p$ is defined by $\r_p=\max(\alpha,\beta_p)$ with $\alpha\in]0,1[$ which is a universal constant and $\beta_p\in]0,1[$ depends on $p\geq1$. In order to prove that $(\r_p)_{p\geq1}$ is non decreasing, it suffices to prove the same for $(\beta_p)_{p\geq1}$.

For $p\geq1$, the definition of $\beta_p$ consists in fixing $\beta_p\in]0,1[$ s.t. for $\beta_p<R<1$ and for all $\v>0$ we have
\begin{equation}\label{LowerBoundGB1}
\frac{1}{\left(\dfrac{1}{\r}-1\right)\displaystyle\int_\r^1t\rho_{\v,p}(t)^{-2}{\rm d}t}\geq\gamma
\end{equation}
where $\rho_{\v,p}$ is defined in \eqref{EqRadEpsSol} and $\gamma>0$ is a constant (the computations are made in \cite{GB1} with {$\gamma=4$}).

Note that it is easy to prove that 
\begin{equation}\label{UnifConvMod}
\rho_{\v,p}\to\rho_{\infty,p}\text{ in }L^\infty([\r,1]) \text{ (when $\v\to\infty$)}
\end{equation}
 with $\rho_{\infty,p}$  defined in \eqref{RadHarmSol}. This uniform convergence is obtained first with the $H^1$ convergence of $u_{\v,p}\to u_{\infty,p}$ (defined in \eqref{EqRadEpsSol}$\&$\eqref{RadHarmSol}). Then using the radially symmetric structure of the function the uniform convergence \eqref{UnifConvMod} follows directly.

Clearly, with the help of \eqref{UnifConvMod} and using the fact that $\rho_{\v,p}\geq\rho_{\infty,p}$ (see Lemma \ref{COmpRhov}), the lower bound \eqref{LowerBoundGB1} holds for all $\v>0$ if and only if 
\begin{equation}\label{LowerBoundGB1-Infty}
\frac{1}{\left(\dfrac{1}{\r}-1\right)\displaystyle\int_\r^1t\rho_{\infty,p}(t)^{-2}{\rm d}t}\geq\gamma.
\end{equation}
We are now in position to get that $(\beta_p)_{p\geq1}$ is non decreasing by proving that for all $r\in[\r,1]$ and $p\geq1$ we have $\rho_{\infty,p+1}(r)\leq \rho_{\infty,p}(r)$.

We fix $r\in[\r,1]$ and we let
\[
\begin{array}{cccc}
f_r:&[1,\infty[&\to&[0,1]
\\
&p&\mapsto& \rho_{\infty,p}(r)=\dfrac{1}{1+\r^p}\left(r^p+\dfrac{\r^p}{r^p}\right)
\end{array}.
\] 
It is clear that $f_r$ is smooth and that
\[
f'_r(p)=\frac{\ln(r)\left[r^p-\left(\dfrac{\r}{r}\right)^p\right](1+\r^p)+\ln(\r)\left[\left(\dfrac{\r}{r}\right)^p-(\r r)^p\right]}{(1+\r^p)^2}.
\]
We have obviously that $f'_r(p)\leq0$ if $\sqrt\r\leq r\leq1$ and if $\r\leq r\leq\sqrt\r$ then letting $r=s\r$ with $s\in[1,\dfrac{1}{\sqrt\r}]$ we have
\[
f'_r(p)=\frac{\ln(\r)(1-\r^p)s^p\r^p+\ln (s)(s^p\r^p-s^{-p})}{(1+\r^p)^2}.
\]
And once agin we have $f'_r(p)\leq0$.

Consequently the function $f_r$ is non increasing, {\it i.e.}, $\rho_{\infty,p+1}(r)\leq \rho_{\infty,p}(r)$. The last inequality imply thus with the help of definition of $\beta_p$ (see \eqref{LowerBoundGB1}) that $\beta_{p+1}\geq\beta_p$. Therefore $\r_{p+1}\geq \r_p$.\\

\noindent {\bf Step 2. }For $p\geq1$, $\r_p<\r<1$ and $\dom=B(0,1)\setminus\overline{B(0,\r)}$, $u_{\infty,p}$ minimizes $m_\infty(p,p)$\\

This step is a direct consequence of Theorem \ref{P-GolBer}, Lemma \ref{LFond1} and \eqref{UnifConvMod}. Indeed from Theorem \ref{P-GolBer}, for $\v>0$, $u_{\v,p}$ defined by \eqref{EqRadEpsSol}$\&$\eqref{ModulusEqEpsRadSol} minimizes $m_\v(p,p)$. 

On the one hand, by \eqref{UnifConvMod}, $u_{\v,p}\to u_{\infty,p}$ in $L^\infty(B(0,1)\setminus\overline{B(0,\r)})$.

On the other hand, with the help of Lemma \ref{LFond1}, up to pass to a subsequence, when $\v\to\infty$, $u_{\v,p}$ converges weakly in $H^1(B(0,1)\setminus\overline{B(0,\r)})$ to a minimizer of $m_\infty(P)$ for some $P\in\A_{p,p}$. 

By combining both previous claims we get that $u_{\infty,p}$ minimizes $m_\infty(p,p)$.\\

\noindent{\bf Step 3.} Conclusion\\

Note that for $p=1$ $m_\infty(1,1)<2\pi$ and thus the result of Proposition \ref{CompaPropBG} is obvious.

We prove that if $p\geq2$, $\r_p<\r<1$ and $\dom=B(0,1)\setminus\overline{B(0,\r)}$ then $m_\infty(p,p)<m_\infty(p-1,p-1)+2\pi$. 

Once this is done, by conformal invariance, we get that if $\dom$ is an annular type domain whose conformal {\it ratio} satisfies $\r_p<\r_\dom<1$ then we have $m_\infty(p,p)<m_\infty(p-1,p-1)+2\pi$.

Let $p\geq2$, $\r_p<\r<1$ and $\dom=B(0,1)\setminus\overline{B(0,\r)}$. From Steps 1$\&$2, we have for $q\in\{p-1,p\}$ that $m_\infty(q,q)$ is reached by $u_{\infty,q}$.

Consequently (using Theorem 1.3 in \cite{HR})
\begin{eqnarray*}
m_\infty(p,p)-m_\infty(p-1,p-1)&=&E_\infty(u_{\infty,p})-E_\infty(u_{\infty,p-1})
\\
&=&2\pi\left[p\frac{1-\r^p}{1+\r^p}-(p-1)\frac{1-\r^{p-1}}{1+\r^{p-1}}\right].
\end{eqnarray*}
Consequently, for $R\in]0,1[$
\begin{eqnarray*}
&E_\infty(u_{\infty,p})-E_\infty(u_{\infty,p-1})<2\pi\\
\Leftrightarrow&p(1-\r^p)(1+\r^{p-1})-(p-1)(1-\r^{p-1})(1+\r^p)<(1+\r^{p-1})(1+\r^p)
\\\Leftrightarrow&Q_p(\r):=p-1-p\r-\r^p<0
\end{eqnarray*}
and 
\begin{eqnarray*}
E_\infty(u_{\infty,p})-E_\infty(u_{\infty,p-1})=2\pi\Longleftrightarrow&Q_p(\r)=0.
\end{eqnarray*}
It is easy to check that, for $p\geq2$ and $\r\in]0,1[$, $Q_p$ is decreasing and that $Q_p(1)=-2$, $Q_p(0)=p-1$. Therefore $Q_p$ admits a unique zero $\tilde{\r}_p$ in $]0,1[$ and for $\r\in]0,1[$ we have $Q_p(\r)<0\Longleftrightarrow \tilde{\r}_p<\r<1$. 

We now prove that $\tilde{\r}_p\leq\r_p$. Let $\r_p<\r<1$. From Steps 1$\&$2, for $q\in\{p-1,p\}$ we have that $m_\infty(q,q,B(0,1)\setminus\overline{B(0,\r)})$ is reached by $u_{\infty,q}$.  Consequently, using Proposition \ref{Property1Meps}.2 we have 
\[
E_\infty(u_{\infty,p})-E_\infty(u_{\infty,p-1})\leq2\pi.
\]
This inequality implies that (from the definitions of $Q_p$ and $\tilde{\r}_p$) $Q_p(\r)\leq0$ and thus $\r\geq\tilde{\r}_p$. Because $\r_p<\r<1$ is arbitrary this consequence proves that $\tilde{\r}_p\leq\r_p$.

The inequality $\tilde{\r}_p\leq\r_p$ expresses that if $\r\in]\r_p,1[$ then $m_\infty(p,p)-m_\infty(p-1,p-1)<2\pi$ and this ends the proof of Proposition \ref{CompaPropBG}.

\end{proof}
\begin{remark}\begin{itemize}
\item {\bf Numerical computation.} Berlyand and Mironescu obtained the existence of Ginzburg-Landau minimizers in $\I_{1,1}$ for large $\e$ without restriction on the capacity of the domain ({\it cf.}  Corollary 5.5. in {\cite{BeMi0}}). In particular they proved that $u_{\infty,1}$ minimizes $m_\infty(1,1)$ for all $\r\in]0,1[$ ({\it cf.} Proposition 5.2. in \cite{BeMi0}).

 For us the first interesting configuration of degrees is $P=(2,2)$. Since $m_\infty(2,2)\leq E_\infty(u_2)$ we obtain that \eqref{ExistKeyHyp} holds if we have:
\begin{equation}\label{ExistKeyHypRelax1}
E_\infty(u_2)=4\pi\dfrac{1-\r^2}{1+\r^2}<2\pi\dfrac{1-\r}{1+\r}+2\pi=m_\infty(1,1)+2\pi.
\end{equation}
Namely \eqref{ExistKeyHypRelax1} implies \eqref{ExistKeyHyp}.

\noindent The study of \eqref{ExistKeyHypRelax1} is easy to do ({\it cf.} \cite{HR} proof of Theorem 5.4.) and gives:
\[
\text{\eqref{ExistKeyHypRelax1} holds if and only if }\r>\sqrt2-1.
\]
Thus if $\r>\sqrt2-1$ then \eqref{ExistKeyHyp} holds and a minimizer of $E_\e$ in $\I_{2,2}$ exists if $\e$ is large enough. 

On the other hand, the radius $R_1$ obtained in \cite{GB1} is ${\rm e}^{\frac{-1}{16\pi^2}}\simeq 0.99$ while $\sqrt2-1\simeq0.41$.

\item{\bf Comparision of Hypotheses.} As explain in Remark 2.14 of \cite{GB1}, the Hypothesis of Theorem \ref{P-GolBer} is artificial : the optimal thickness condition should depend on $\v$. 

The formulation of Theorem \ref{PropEXist1} is not optimal in the sense given by Berlyand and Golovaty in Remark 2.14 of \cite{GB1}. But it allows to have existence of minimizers for $m_\v(p,p)$ for a wider class of annular type domains:
\begin{enumerate}[$\bullet$]
\item Theorem \ref{PropEXist1} holds for annular type domain while the work of Golovaty and Berlyand is specific to annulars.
\item Proposition \ref{CompaPropBG} means that if the hypothesis on the size of the annular in Theorem \ref{P-GolBer} holds then Hypothesis of  Theorem \ref{PropEXist1} holds. 
\end{enumerate}
\end{itemize}
\end{remark} 

\subsection{Asymptotic behavior of minimizers as $\e \rightarrow +\infty$}
\begin{prop}\label{conv1}
Let $p\geq 1$ be an integer and let $\dom$ be an annular type domain s.t. $m_\infty(p,p) < m_\infty(p-1,p-1)$. Thanks to that condition minimizers $u_\e$ of $E_\e$ in $\I_{p,p}$ do exist for $\e$ large [Prop. \ref{P-ExistPartThm1}] and $\e=+\infty$ [Prop. \ref{compactinfty}]. 

Then it holds that, up to a subsequence,
\begin{equation}
 u_\e \rightarrow u_\infty \ \ \text{in} \ C^{l} \ \text{for all} \ l\in  \mathbb{N},
\end{equation}
\noindent where $u_\infty$ is a minimizer of $E_\infty$ in $\I_{p,p}$.
\end{prop}

\noindent The starting point of the proof of the previous proposition is the following:

\begin{lem}\label{conv0}
Under the same hypothesis as in Proposition \ref{conv1}, we have that, up to a subsequence,
\begin{equation}
u_\e \rightarrow u_\infty \ \ \text{strongly in} \ H^1(\dom) \ \text{and in} \ C^l_{\rm{loc}} \ \forall l\in \mathbb{N}
\end{equation}
\noindent where $u_\infty$ is a minimizer of $E_\infty$ in $\I_{p,p}$.
\end{lem}

\begin{proof}
For $\e$ large, if the domain $\dom$ is s.t. $m_\infty(p,p) < m_\infty(p-1,p-1)$, denoting by $u_\e$ a minimizer of $E_\e$ in $\I_{p,p}$ and by $\tilde{u}_\infty$ a minimizer of $E_\infty$ in $\I_{p,p}$ we have:
\begin{eqnarray}\nonumber
E_\infty(\tilde{u}_\infty)& \leq  & E_\infty(u_\e) \\
& \leq & E_\e(u_\e) \nonumber \\
& \leq & E_\e(\tilde{u}_\infty) \nonumber \\\nonumber
&= & E_\infty(\tilde{u}_\infty) +\frac{1}{4\e^2}\int_\dom (1-|\tilde{u}_\infty|^2)^2. 
\end{eqnarray}
Hence we see that $(u_\e)_\e$ is a minimizing sequence for $m_\infty(p,p)$. By Proposition \ref{compactinfty}, along a subsequence we then have $u_\e \to u_\infty$ in $H^1(\dom)$ for some $u_\infty$ which solves $m_\infty(p,p)$.

The $C^l_{\text{loc}}$ convergence for all $l\in \mathbb{N}$ is obtained by classic elliptic estimates (see \cite{GilbargTrudinger}).
\end{proof}

We now prove that the convergence holds in $C^{l}(\overline{\dom})$ for all $l\in \mathbb{N}$. To this end we adapt the strategy of Berlyand and Mironescu (Section 8 in \cite{BeMi0}).

We divide the proof into four steps:\\

\noindent \textbf{Step 1.} We have that $|u_\e|$ is uniformly close to $1$ near $\p\dom$  for large $\e$
\begin{lem} \label{closenessto1}
Let $\rho_\e:=|u_\e|$. For all $\eta >0$, there exist $\delta>0$ and $\e_0>0$ s.t. for all $\e \geq \e_0$ and for all $z$ s.t. $\dist(z,\partial \dom) <\delta$ it holds that 
$$\| \rho_\e -1 \|_{L^\infty} < \eta.$$
\end{lem}
\noindent For the proof of this lemma we need the following reformulation of Berlyand$\&$Mironescu (see Lemma 8.3 in \cite{BeMi0}) of a result of Brezis$\&$Nirenberg :

\begin{lem}[Theorem A3.2. in \cite{BrezisNirenberg2}]\label{VMO}
Let $(g_n) \subset \text{VMO}(\partial \dom; \mathbb{S}^1)$ be s.t. $g_n \rightarrow g$ strongly in $\text{VMO}(\partial \dom)$. Then for each $0<a<1$, there is some {$\delta'>0$}, independent of $n$ s.t. 
$$a \leq |\tilde{u}(g_n)(z)| \leq 1, \ \ \text{if} \ \ \dist(z,\partial \dom) < \delta'.$$
Here $\tilde{u}(g_n)$ is the harmonic extension of $g_n$ to $\dom$.
\end{lem}

\begin{proof}[Proof of Lemma \ref{closenessto1}]
Let $u_\e$ be a minimizer of $E_\e$ in $\I_{p,p}$ for $\e$ large enough. We write
$u_\e=v_\e+w_\e$ with $w_\e$ which satisfies
\begin{equation}
\left\{
\begin{array}{rclll}
-\Delta w_\e&=&\dfrac{1}{\e^2}u_\e(1-|u_\e|^2)\ & \text{in} \ \dom \\
w_\e &=&0 \ &\text{on} \ \partial \dom
\end{array}
\right.
\end{equation}
and $v_\e$ the harmonic extension of $\tr _{\partial \dom}{u_\e}$, {\it i.e.},
\begin{equation}
\left\{
\begin{array}{rclll}
\Delta v_\e &=&0 \ & \text{in} \ \dom \\
v_\e &=&u_\e \ & \text{on} \ \dom
\end{array}.
\right.
\end{equation}
To estimate $\|\nabla w_\e \|_{L^\infty(\dom)}$ we use the standard elliptic estimate 

\begin{lem}[Lemma A.2. in \cite{BBH0}]\label{ellipticestimate}
Let $w\in C^2(\overline{\dom})$ satisfy 
\begin{equation}
\left\{
\begin{array}{rcll}
\Delta w &=& f \ \text{in} \ \dom \\
w & = & 0 \ \text{on} \ \partial \dom
\end{array}
\right..
\end{equation}
\noindent Then, for some constant $C_\dom$ depending only on $\dom$, we have:
\begin{equation}
\|\nabla w \|_{L^\infty(\dom)} \leq C_\dom \|w\|_{L^\infty(\dom)}^{1/2} \|f\|_{L^\infty(\dom)}^{1/2}.
\end{equation}
\end{lem}

Thanks to  Lemma \ref{ellipticestimate} we obtain [note that $\|w_\v\|_{L^\infty(\dom)}\leq\|u_\v\|_{L^\infty(\dom)}+\|v_\v\|_{L^\infty(\dom)}\leq2$]
\begin{equation}
\|\nabla w_\e \|_{L^\infty(\dom)} \leq \sqrt{2}C_\dom \times \frac{1}{\e} 
\end{equation}
where $C_\dom$ is a constant  depending only on $\dom$. 

Thus, since $w_\e=0$ on $\partial \dom$ we obtain that there exists a constant $C'_\dom$ s.t. 
\begin{equation}
|w_\e(z)| \leq C'_\dom\frac{1}{\e}\dist(z,\partial \dom).
\end{equation}
We note that, up to a subsequence, $ \tr_{\partial \dom}{u_\e} \rightarrow \tr_{\partial \dom}{u_\infty}$ strongly in $H^{1/2}(\partial \dom)$ because $ {u_\e} \rightarrow u_\infty$ strongly in $H^1(\dom)$. Since $H^{1/2}{\hookrightarrow} \text{VMO}$ in 1D we can apply Lemma \ref{VMO} to obtain that for all $\eta >0$ there exists $\delta'$ and $\e_0$ s.t. for all $\e \geq \e_0$
$$1-\frac{\eta}{2} \leq |v_\e| \leq  1, \ \ \text{if} \ \dist(z,\partial \dom)<\delta'.$$
Hence we find that 
\begin{eqnarray}
 1 \geq |u_\e(z)| & \geq &|v_\e(z)|-|w_\e(z)| \nonumber \\
 & \geq & 1-\frac{\eta}{2}- C'_\dom\frac{1}{\e}\dist(z,\partial \dom), \ \text{if} \ \e \geq \e_0 \ \text{and} \ \dist(z,\partial \dom) < \delta' \nonumber \\
 & \geq & 1-\eta \nonumber
\end{eqnarray}
if $\dist(z,\partial \dom) < \delta:= \min\{ \delta',\displaystyle{ \frac{\eta\e_0}{C'_\dom}\}}$.
\end{proof}

\noindent{\bf Step 2.} Lifting close to $\p\dom$\\

Now thanks to Lemma \ref{closenessto1} we know that,  for some $\delta >0$ and for sufficiently large $\v$, $u_\e$ does not vanish in 
\[
\dom^+_\delta:=\{z \in \dom\,|\,\dist(z,\partial \Omega)<\delta\}
\]
 nor in 
 \[
 \dom^{-}_\delta:=\{ z \in \dom \,| \,\dist(z,\partial \omega)<\delta\}.
 \]
  We set $\rho_\e =|u_\e|$ and $\rho_\infty=|u_\infty|$. Note that up to consider a smaller value for $\delta$ we may assume that $|u_\infty|\geq 1-\eta$ in $\dom^+_\delta\cup\dom^-_\delta$ (because $u_\infty$ is smooth in $\overline{\dom}$, see Lemma 4.4 \cite{BeMi0}).
  
 Therefore we can write $u_\infty=\rho_\infty e^{\imath \varphi}$, where $\varphi$ is a locally defined harmonic function and $\nabla \varphi$ is globally defined. 
 
 In $\dom^{+}_\delta$ we have that 
 \[
 \deg\left(\frac{u_\e}{u_\infty},\partial \Omega\right)=0\text{ and }\deg\left(\frac{|u_\infty|u_\e}{|u_\e|u_\infty},\partial \dom^+_\delta \setminus \partial \Omega\right)=0.
\] We can thus find $\psi_\e \in H^1(\dom^+_\delta,\R)$ s.t.  $u_\e=\rho_\e e^{\imath(\varphi +\psi_\e)}=\rho e^{\imath(\varphi + \psi)}$ in $\dom^+_\delta$. The same is true in $\dom^-_\delta$. In $\dom^{\pm}_\delta$ the Ginzburg-Landau equation is then equivalent to the following equations on $\rho$ and $\psi$:


\begin{equation}\label{rhoequation}
\left\{
\begin{array}{rclll}
-\Delta \rho&=&\frac{1}{\e^2}\rho(1-\rho^2)-\rho|\nabla (\varphi +\psi)|^2 & \text{in} \ \dom^{\pm}_\delta \\
\rho&=&1 & \text{on} \ \partial \dom \\
\end{array}
\right.,
\end{equation}

\begin{equation}\label{phiequation}
\left\{
\begin{array}{rclll}
-\dive(\rho^2\nabla \psi)&= &\dive(\rho^2 \nabla \varphi)= &2\rho\nabla \rho\cdot \nabla \varphi& \text{in} \ \dom \\
{\partial_\nu \psi}& =&0 & & \text{on} \ \partial\dom
\end{array}
\right..
\end{equation}
Note that the last equation can be rewritten as 
\begin{equation}\label{phiequation}
\begin{array}{rclll}
\Delta \psi & =&\dive[(1-\rho^2)\nabla (\varphi +\psi) ] & \text{in} \ \dom^{\pm}_\delta.
\end{array}
\end{equation}

\noindent \textbf{Step 3.} $\nabla \psi_\e$ is bounded in $L^4(\dom^{\pm}_\delta)$\\

Fix $z_0\in \partial \dom$. In order to simplify the proof we assume that $z_0=0$, $\dom \subset \{z ; \text{Im}(z) >0 \}$ and $\partial \dom \subset \R$ in a neighborhood $U$ of $z_0$. (These assumptions are not essential for carrying out the arguments below but make the redaction easier). Let $r>0$ to be determined later s.t. $B_r:=B(0,r) \subset U$. Using the Schwarz reflection we extend $\rho, \psi$ and $F=(1-\rho^2)\nabla \psi=(F_1,F_2)$ to $B_r \setminus \overline{\dom}$ by setting for $z \in B_r \setminus \overline{\dom}$
$$\tilde{\rho}(z)=\rho(\overline{z}), \ \ \tilde{\psi}(z)=\psi(\overline{z}), \ \ \tilde{F}(z)=(F_1(\overline{z}),-F_2(\overline{z})).$$
We can then show that $\tilde{\psi}$ is a solution of 
\begin{equation}
\Delta \tilde{\psi}=\dive \tilde{F}(z) \ \ \ \text{in} \ B_r.
\end{equation}
By standard elliptic estimates (see Theorem 7.1 in \cite{giaquinta2013introduction}), we have
\begin{equation}
\|\nabla \tilde{\psi}\|_{L^4(D_r)} \leq C_4\left(  \|\tra_{\partial B_r} \tilde{\psi}\|_{W^{1-\frac{1}{4},4}(\partial B_r)} +\|\tilde{F}\|_{L^4(B_r)} \right).
\end{equation}
By scaling, the constant $C_4$ does not depend on $r$. We also have that 
\begin{eqnarray}
\|\tilde{F}\|_{L^4(B_r)} & \leq & \| 1-\tilde{\rho}\|_{L^\infty(B_r)} \| \nabla \tilde{\psi} \|_{L^4(B_r)}. 
\end{eqnarray}

\noindent Thanks to Lemma \ref{closenessto1} we can choose $r$ small enough s.t.  for  $\e$ large enough we have $\| 1-\tilde{\rho}\|_{L^\infty(B_r)}<\frac{1}{2C_4}$. Hence we obtain that 
\begin{equation}
\|\nabla \tilde{\psi}\|_{L^4(B_r)} \leq 2C_4 \|\tra_{\partial B_r} \tilde{\psi}\|_{W^{1-\frac{1}{4},4}(\partial B_r)}.
\end{equation}
\noindent We can prove that, for $r$ small enough and along a subsequence we have $\tra_{\partial B_r} \tilde{\psi}$ is bounded in $W^{1-\frac{1}{4},4}(\partial B_r)$. Indeed, along a subsequence,
$\tra_{\partial B_r \cap \dom} \tilde{\psi}$ is bounded in $H^1(\partial B_r \cap \dom) $ for some $r>0$ s.t. $B_r \subset U$ thanks to the coarea formula and to the fact that $ \psi$ is bounded in $H^1(\dom)$ (since $|\nabla \psi| \leq |\nabla u_\e|$ in $\dom$). Using the [continuous] Sobolev injection $H^1(\partial B_r) \hookrightarrow W^{1-\frac{1}{4},4}(\partial B_r)$ we obtain the result. 
Thus (up to a subsequence) $\| \nabla \psi _\e \|_{L^4(B_r \cap \dom)}$ is bounded for $r$ small enough. \\

\noindent  Repeating the previous argument we 
find that: for all $z_0 \in \partial \dom$ there exist $r_{z_0}>0$ and $M_{z_0}>0$ s.t.  (up to a subsequence) $\| \nabla \psi _\e \|_{L^4(B_{r_{z_0}} \cap \dom)}\leq M_{z_0}$.

\noindent Thanks to the fact that $\partial \dom$ is compact we deduce that there exist $\delta_1>0$, a subsequence and $M$ s.t., letting $\dom_{\delta_1}=\{z \in \dom \ | \ \dist(z,\partial \dom) <\delta_1\}$, we have 
$$\|\nabla \psi_\e \|_{L^4(\dom_{\delta_1})} \leq M, \ \ \text{for }  \e \ \text{large enough}. $$
($M$ is independent of $\e$) 

\noindent Now since $u_\e \rightarrow u_\infty$ in $C^l_{\text{loc}}$ for all $l\in \mathbb{N}$ we obtain that $\nabla \psi_\e$ is bounded in $L^4(\dom^+_\delta)$ and in $L^4(\dom^-_\delta)$.\\

\noindent \textbf{Step 4.} Elliptic estimates and a bootstrap argument\\

We work in $\dom^+_\delta$ but the argument is the same for $\dom^-_\delta$. We can use the equation satisfied by $\rho_\e$ \eqref{rhoequation}, the fact that $\nabla \varphi$ is bounded in $L^\infty$ (see Lemma 4.4 in \cite{BeMi0}) and the previous step to obtain that $\Delta \rho_\e$ 
is bounded in $L^2(\dom^{+}_\delta)$. Hence the elliptic regularity implies that $\rho_\e$ is 
bounded in $W^{2,2}(\dom^{+}_{\delta/2})$. Indeed one can multiply $\rho$ by a cut-off 
function $\chi \in C^\infty(\dom^+_\delta)$ s.t. $\chi \equiv 1$ in $\dom^{+}
_{{\delta}/{2}}$ and $\chi=0$ on $\partial \dom^+_\delta \setminus \partial \Omega$. We can 
then see that $\Delta (\chi \rho)$ is bounded in $L^2(\dom^{+}_\delta)$ and since the boundary 
conditions are adapted to global regularity we deduce that $\chi \rho$ is bounded in $W^{2,2}
(\dom^+_\delta)$. Using the fact that $\chi \equiv 1$ in $\dom^+_{{\delta}/{2}}$ we obtain 
the result. Now since $W^{1,2} \underset{{\rm cont}}{\hookrightarrow} L^p$ for all $1<p<+\infty$ we have that $\nabla 
\rho$ is bounded in $L^p(\dom^+_{\delta /2})$ for all $1<p<+\infty$. \\

\noindent We now use the equation satisfied by $\psi_\e$, written as 
\begin{equation}
\Delta \psi_\e =\frac{2}{\rho_\e} \nabla \rho_\e \cdot \nabla (\psi_\e +\varphi).
\end{equation}
We note that ${1}/{\rho_\e}$ and $\nabla \varphi$ are bounded in $L^\infty(\dom^+_\delta)$ and we deduce that $\Delta \psi_\e$ is bounded in $L^q(\dom^+_\delta)$ for all $1<q<+\infty$. Hence using a similar argument as before with a cut-off function we can show that $\psi_\e$ is bounded in $W^{2,q}(\dom^+_{\delta/2})$ for all $1<q<+\infty$. In particular $\nabla \psi_\e$ is bounded in $W^{1,q}(\dom^+_{\delta/2})$ for all $1<q<+\infty$. Using the fact that $W^{1,q} \cap L^\infty$ is an algebra (see \textit{e.g.} Proposition 9.4 p.269 in  \cite{Functionalanalysis}) we find that $\Delta \rho_\e$ is bounded in $W^{1,q}(\dom^+_{\delta/2})$ for all $1<q<+\infty$ and thus $\rho_\e$ is bounded in $W^{3,q}(\dom^+_{\delta/2})$. By  a straightforward induction we obtain that  $$\rho,\psi, \  \text{are bounded in} \ W^{m,q}(\dom^+_{\delta/2}) \ \text{for all } m\geq 2, \ 1<q<+\infty.$$

\noindent Thanks to Sobolev injections for any $l \in \mathbb{N}$ and any $0<\gamma<1$ we can choose $m \geq 1$ 
and $1<q<+\infty$ s.t. $k=m-1$ and $1-\frac{2}{q}> \beta$ we then have $W^{m,q} \hookrightarrow C^{l,\gamma}
(\overline{\dom^+_{\delta/2}})$ and this embedding is compact. We thus have that, up to a subsequence, $ 
u_\e=\rho_\e e^{\imath (\varphi +\psi_\e)} \rightarrow u$ in $C^{l,\gamma}$ for some $u$ as $\e \rightarrow 
\infty$ in $\dom^+_{\delta/2}$. But by Lemma \ref{conv0} we have $u=u_\infty$. Using the $C^l_\text{loc}$ 
convergence, we can finally conclude that $u_\e \rightarrow u$ in $C^l(\overline{\dom})$ for all $l \in 
\mathbb{N}$. 

\section{Non Existence Result}\label{ProofThmNonExist}
This section is dedicated to the proof of Theorem \ref{ThmNonExist}. We fix $p,q\in\N^*$, $p\neq q$. For the simplicity of the presentation we assume that $p>q$. The case $p<q$ is similar. 

We adapt here the strategy of Misiats [used to prove Theorem 2 in \cite{Misiats1}].

We denote $d:=p-q\in\N^*$ and $\Ar:=B(0,1)\setminus\overline{B(0,\r)}$ where $R\in]0,1[$. We are going to prove that for $R$ sufficiently close to $1$ and large $\v$ there is no minimizer for $m_\v(p,q)$. 
\subsection{Strategy of the proof}

By Theorem \ref{P-GolBer}, there is $\r^{(1)}_{q}$ [$\r^{(1)}_{q}$ is independent of $\v$] s.t. $m_\v(q,q,\Ar)$ is attained by the radial Ginzburg-Landau solution $u_\v=\rho_\v{\rm e}^{\imath q\theta}$ [here $\rho_\v=\rho_{\v,q}$ depends also on $q$ see \eqref{EqRadEpsSol}\&\eqref{ModulusEqEpsRadSol}].

Because $\rho_\v>0$ in $\Ar$, it is easy to see that 
\[
\J_{p,q}=\{\rho_\v w\,|\, w\in \J_{p,q}\}.
\]
Thus we have
\begin{equation}\label{NonExs1}
m_\v(p,q)=\inf_{w\in\J_{p,q}}E_\v(\rho_\v w).
\end{equation}
By Lemma 21 in \cite{BeRy}, we have for $w\in\J$
\begin{equation}\label{NonExs2}
E_\v(\rho_\v w)=E_\v(u_\v)+L_\v(w)
\end{equation}
with
\begin{equation}\label{NonExs3}
L_\v(w)=\dfrac{1}{2}\int_\Ar\rho_\v^2|\n w|^2-{q^2}{}\rho_\v^2|\n\theta|^2|w|^2+\dfrac{1}{2\v^2}\rho_\v^2(1-|w|^2)^2.
\end{equation}
By combining \eqref{NonExs1}, \eqref{NonExs2} and \eqref{NonExs3} we get
\begin{equation}\label{NonExs4}
m_\v(p,q)=E_\v(u_\v)+\inf_{w\in\J_{p,q}}L_\v(w).
\end{equation}
We argue by contradiction and we assume that 
$$\text{there is} \ \v=\v_n\uparrow\infty \ \text{s.t.} \ m_\v(p,q) \ \text{is attained by} \ \rho_\v w_\v.$$

Our strategy consists in proving that for $\r$ sufficiently close to $1$, we have
\begin{equation}\label{NonExs5}
L_\v(w_\v)>d\pi.
\end{equation}
Estimate \eqref{NonExs5} with \eqref{NonExs4}\&Proposition \ref{Property1Meps}.2 implies that $m_\v(p,q)>m_\v(q,q)+d\pi$  which is in contradiction with Proposition \ref{Property1Meps}.2.\\


The key argument is a minoration of $L_\v(w_\v)$ by a sum of infinitely many {\it infima} of functional (see \eqref{DecSumTilL}). These functionals have the form $|a_k|^2F_k(\cdot)$ where the $a_k$'s are the Fourier coefficients of $\tr_{\S^1}(w{\rm e}^{-\imath q\theta})$.  The $F_k$'s are defined in $H^1(]\r,1[,\C)$ and we imposed Dirichlet boundary condition for $r=1$ whereas we let the other boundary $r=\r$ free (see \eqref{MinProM}). Note that since the boundary $r=\r$ is free we obtained homogeneous Neumann boundary condition for $r=\r$.

By using some properties of $(a_k)_k\in\C^\Z$ we apply Lemma 3 in \cite{Misiats1}  (see  Proposition \ref{PropCoefV}.3 below) in order to obtain that for large $\v$ we have \eqref{NonExs5}.

\subsection{Asymptotic analysis of $v_\v=w_\v{\rm e}^{-\imath q\theta}$}
The goal of this subsection is to prove that $\tr_{\S^1} (w_\v{\rm e}^{-\imath q\theta})\to 1$ in $L^2(\S^1)$.\\

 By Lemma \ref{LFond1}, up to pass to a further subsequence, there is $P\in\A_{(p,q)}$ and $u_\infty\in\J_P$ s.t. $\rho_\v w_\v\weak u_\infty$ in $H^1$. Moreover $u_\infty$ minimizes $m_\infty(P)$,
 \begin{equation}\label{ContradAlterna}
 m_\infty(p,q)=m_\infty(P)+\pi|P-(p,q)|,
\end{equation}
and we have $P=(q',q')$ for some $0\leq q' \leq q$ from Proposition \ref{PropEpsInftHR}. However for $R> R_q^{(1)}$ we have that $q'=q$. Indeed, recall that for $R> R_q^{(1)}$, $m_\infty(q,q)$ is uniquely attained by the radial harmonic map and, according to the discussion in Section 3.3 it holds that for all $0\leq r<q$ we have 
$$m_\infty(q,q)< m_\infty(r,r)+2\pi(q-r).$$
But if $q'<q$ then we find that (using Lemma \ref{LFond1}) 
\[
m_\infty(p,q)=m_\infty(q',q')+\pi(p-q)+2\pi(q-q') <m_\infty(q,q)+\pi(p-q)
\]
 which is in contradiction with Proposition \ref{PropEpsInftHR}.

%

Consequently, up to multiply by a constant of $\S^1$, we have that $u_\infty=u_{\infty,q}$ (defined in \eqref{RadHarmSol}) where
\[
u_{\infty,q}(r{\rm e}^{\imath\theta})=\frac{1}{1+\r^q}\left(r^q+\frac{\r^q}{r^q}\right){\rm e}^{\imath q\theta}.
\]

We now write $w_\v\in\J_{p,q}$ as $w_\v=v_\v{\rm e}^{\imath q\theta}$ with $v_\v\in\J_{d,0}$. 

\noindent From the previous arguments we know that $\rho_\v w_\v=\rho_\v v_\v{\rm e}^{\imath q\theta}\weak u_{\infty,q}=\rho_{q}{\rm e}^{\imath q\theta}$ in $H^1(\dom)$ [here we write $\rho_q$ instead of $\rho_{\infty,q}$].

\noindent Moreover, from Lemma \ref{conv0}, we have $\rho_\v {\rm e}^{\imath q\theta}\to \rho_{q}{\rm e}^{\imath q\theta}$ in $H^1(\dom)$.  Consequently $v_\v\weak1$ in $H^1(\dom)$. Therefore $\tr_{\S^1} v_\v\to 1$ in $L^2(\S^1)$.
\subsection{Reformulation of $L_\v(w_\v)$ and a minoration of $L_\v(w_\v)$}

In order to get a nice lower bound for $L_\v(w_\v)$ we first reformulate $L_\v(w_\v)$.\\

The argument is based on the Fourier expansion of $\tr_{\S^1} v_\v$:
\[
\tr_{\S^1} v_\v({\rm e}^{\imath\theta})=\sum_{k\in\Z}a_k {\rm e}^{\imath k\theta}.
\]
We have the following proposition:
\begin{prop}\label{PropCoefV}
\begin{enumerate}
\item $\displaystyle\sum_{k\in\Z}k|a_k|^2=d$.
\item $\displaystyle\sum_{k\in\Z^*}|a_k|^2\to0$ when $\v\to\infty$.
\item Let $k_0\in\N^*$, there is $C_1$ (depending only on $k_0$) and a sequence $c_\v>0$ (depending only on $k_0$ and $\v$) s.t. $c_\v\to1$ when $\v\to\infty$ satisfying for $k=1,...,k_0$.
\[
|a_k|\leq c_\v|a_{-k}|+C_1\sum_{l\in\Z^*}|a_l|^2.
\]
\end{enumerate}
\end{prop}
\begin{proof}
The first assertion is the degree formula. The second assertion comes from the convergence $\tr_{\S^1} v_\v\to1$ in $L^2$. The third assertion is Corollary 2 in \cite{Misiats1} by noting that Lemma 3 in \cite{Misiats1} holds.
\end{proof}

We now go back to the $L_\v$ functional. Writing $w_\v=v_\v{\rm e}^{\imath q\theta}$ we have
\begin{eqnarray*}
L_\v(v{\rm e}^{\imath q\theta})&=&\dfrac{1}{2}\int_\Ar\rho_\v^2\left[q^2|\n\theta|^2|v|^2+|\n v|^2+2q\n\theta\cdot(v\wedge\n v)\right]-\\&&\phantom{aaaaaaassssssss}-{q^2}{}\rho_\v^2|\n\theta|^2|v|^2+\dfrac{1}{2\v^2}\rho_\v^2(1-|v|^2)^2
\\
&=&\dfrac{1}{2}\int_\Ar\rho_\v^2|\n v|^2+2q\rho_\v^2\n\theta\cdot(v\wedge\n v)+\dfrac{1}{2\v^2}\rho_\v^2(1-|v|^2)^2
\\&=:&\tilde{L}_\v(v)+\dfrac{1}{4\v^2}\int_\Ar\rho_\v^2(1-|v|^2)^2.
\end{eqnarray*}
We now focus on the $\tilde{L}_\v$ functional and we prove that for sufficiently large $\v$ and for $\r$ sufficiently close to $1$, we have
\begin{equation}\label{NonExs6}
 L_\v(w_\v)\geq \tilde{L}_\v(v_\v)>d\pi.
\end{equation}
To prove \eqref{NonExs6} we switch to polar coordinates (with a little abuse of notation) and we write
\[
v_\v(r,\theta)=\sum_{k\in\Z}a_k f_k(r){\rm e}^{\imath k\theta},\,r\in]\r,1[,\,\theta\in]0,2\pi[
\]
where $f_k\in H^1(]\r,1[,\C)$ is s.t. $f_k(1)=1$.

Note that the map $\rho_\v$ depends only on $r\in]\r,1[$. Therefore we have the following expansion:
\[
\tilde{L}_\v\left(\sum_{k\in\Z}a_kf_k(r){\rm e}^{\imath k\theta}\right)=\pi\sum_{k\in\Z}|a_k|^2\int_\r^1\rho_\v^2\left[r|f_k'|^2+\dfrac{k^2+2qk}{r}|f_k|^2\right].
\]
For $k\in\Z$, and $f\in H^1(]\r,1[,\C)$, we let 
\[
F_k(f)=\int_\r^1\rho_\v^2\left[r|f'|^2+\dfrac{k^2+2qk}{r}|f|^2\right]
\]
and 
\begin{equation}\label{MinProM}
m_k=\inf \left\{F_k(f) \ | \ f \in H^1(]R,1[,\C) \ \text{s.t.} \ f(1)=1\right\}
\end{equation}
\subsection{Minoration of $ \tilde{L}_\v(v_\v)$}
It is clear that we have
\begin{equation}\label{DecSumTilL}
L_\v(w_\v)\geq\tilde{L}_\v\left(\sum_{k\in\Z}a_kf_k(r){\rm e}^{\imath k\theta}\right)\geq \pi\sum_{k\in\Z}a_km_k.
\end{equation}
In order to get a lower bound for $m_k$ we use the following lemma:
\begin{lem}\label{COmpRhov}
For $\v>0$ we have $\rho_\v\geq\rho_q$ where $\rho_q(r)=\dfrac{1}{1+\r^q}\left(r^q+\dfrac{\r^q}{r^q}\right)$.
\end{lem}
\begin{proof}
Let $\v>0$ and let $U=\{x\in\Ar\,|\,\rho_\v(x)<\rho_q(x)\}$. We argue by contradiction and we assume that $U\neq\emptyset$. Note that $U$ is a smooth open set and that $\tr_{\p U}(\rho_\v{\rm e}^{\imath q\theta})=\tr_{\p U}(\rho_q{\rm e}^{\imath q\theta})$. 

\noindent By the minimality of $\rho_q{\rm e}^{\imath q\theta}$ we have 
\[
E_\infty(\rho_q{\rm e}^{\imath q\theta},U)\leq E_\infty(\rho_\v{\rm e}^{\imath q\theta},U).
\]

\noindent On the other hand, by the definition of $U$ and because $0\leq\rho_\v,\rho_q\leq1$ we have
\[
\int_U(1-\rho_q^2)^2<\int_U(1-\rho_\v^2)^2.
\]
Consequently 
\[
E_\v(\rho_q{\rm e}^{\imath q\theta},U)< E_\v(\rho_\v{\rm e}^{\imath q\theta},U)
\]
and this is in contradiction with the minimality of $\rho_\v{\rm e}^{\imath q\theta}$.
\end{proof}
From Lemma \ref{COmpRhov}, for $f\in H^1(]\r,1[,\C)$
\[
F_k(f)\geq\begin{cases}\displaystyle\int_\r^1\rho_q^2\left[r|f'|^2+\dfrac{k^2+2qk}{r}|f|^2\right]&\text{if }k^2+2qk>0\\\displaystyle\int_\r^1\rho_q^2r|f'|^2+\dfrac{k^2+2qk}{r}|f|^2&\text{if }k^2+2qk\leq0\end{cases}.
\]
We let
\[
\rmin=\min_{[\r,1]}\rho_q=\dfrac{2\r^{q/2}}{1+\r^q}.
\]
In order to get \eqref{NonExs6}, it suffices to replace the minimization problem $m_k$ [define in \eqref{MinProM}] by $\tilde{m}_k$ where:
\begin{itemize}
\item[$\bullet$]for $k\leq0\,\&\,k^2+2qk>0$   
\[
\tilde{m}_k=\rmin^2\inf \left.\left\{ \int_\r^1r|f'|^2+\dfrac{k^2+2qk}{r}|f|^2 \ \right| \ f\in H^1(]\r,1[,\C) \ \text{s.t.}f(1)=1\right \}
\]
\item[$\bullet$]for $k\leq0\,\&\,k^2+2qk\leq0$   
\[
\tilde{m}_k=\rmin^2\inf \left.\left\{\int_\r^1r|f'|^2+\dfrac{k^2+2qk}{r\rmin^2}|f|^2 \ \right| \  f\in H^1(]\r,1[,\C) \ \text{s.t.}f(1)=1\right\}
\]
\item[$\bullet$]for $k>0$, $\tilde{m}_k=\dfrac{1}{(1+\r^q)^2}\left[\tilde{m}_k^{(1)}+2\r^q\tilde{m}_k^{(2)}+\r^{2q}\tilde{m}_k^{(3)}\right] $ where
\[
\tilde{m}_k^{(1)}=\inf \left.\left\{\int_\r^1r^{2q+1}|f'|^2+r^{2q-1}{(k^2+2qk)}{}|f|^2 \ \right| \begin{array}{l} f\in H^1(]\r,1[,\C) \\ \text{ s.t.}f(1)=1\end{array}\right\},
\]
\[
\tilde{m}_k^{(2)}=\inf \left.\left\{\int_\r^1r|f'|^2+\dfrac{k^2+2qk}{r}|f|^2 \ \right| \begin{array}{l} f\in H^1(]\r,1[,\C) \\ \text{ s.t.}f(1)=1\end{array}\right\},
\]
\[
\tilde{m}_k^{(3)}=\inf \left.\left\{\int_\r^1r^{-2q+1}|f'|^2+r^{-2q-1}{(k^2+2qk)}{}|f|^2 \ \right|\begin{array}{l} f\in H^1(]\r,1[,\C) \\ \text{ s.t.}f(1)=1\end{array}\right\}.
\]
\end{itemize}
We first study the cases $k\leq0$. According to the definition of $\tilde m_k$    
we divide the presentation in two parts: $k^2+2qk>0$ and $k^2+2qk\leq0$.

It is clear that $k^2+2qk\leq0\Leftrightarrow k=-2q,...,0$.
We treat the case $k^2+2qk>0$\&$k\leq0$, {\it i.e.}, $k<-2q$. \\

\noindent{\bf Case I. $k<-2q$}\\

If $k<-2q$, it is obvious that 
\begin{equation}\label{Mk0Est}
\tilde{m}_k>0,
\end{equation}
and this estimate is sufficient for our argument.\\

\noindent{\bf Case II. $k=-2q,...,0$}\\

We now consider the case: $k=-2q,...,0$. We claim that $k^2+2qk\geq-q^2$. Therefore, by a Poincaré type inequality, there is $1>\r_q^{(2)}\geq \r_q^{(1)}$ s.t. for $\r_q^{(2)}<\r<1$ 
\[
\inf_{\substack{f\in H^1(]\r,1[,\C)\\\text{s.t.}f(1)=1}}\int_\r^1\left[r|f'|^2+\dfrac{k^2+2qk}{r\rmin^2}|f|^2\right]>-\infty.
\]
Therefore, by direct minimization, the {\it infimum} is reached. One can prove that the minimizer of $\tilde{m}_k$ is unique and, letting $\alpha:= \dfrac{k^2+2qk}{\rmin^2}$, it satisfies:
\[
\begin{cases}-(rf')'+\dfrac{\alpha}{r}f=0&\text{for }r\in]\r,1[
\\f(1)=1\&f'(\r)=0
\end{cases}.
\]
By solving the ordinary differential equation we get that 
\[
f_0(r)=A\cos(\sqrt{-\alpha}\ln r)+B\sin(\sqrt{-\alpha}\ln r).
\]
With the boundary conditions we obtain
\[
f_0(r)=\cos(\sqrt{-\alpha}\ln r)+\tan(\sqrt{-\alpha}\ln \r)\times\sin(\sqrt{-\alpha}\ln r).
\]
By using an integration by part we easily get that
\begin{eqnarray*}
\inf_{\substack{f\in H^1(]\r,1[,\C)\\\text{s.t.}f(1)=1}}\int_\r^1\left[r|f'|^2+\dfrac{k^2+2qk}{r\rmin^2}|f|^2\right]&=&f_0'(1)f_0(1)-f'_0(\r)f_0(\r)
\\&=&f'_0(1)=\sqrt{-\alpha}\tan(\sqrt{-\alpha}\ln \r).
\end{eqnarray*}
Thus, if $k=-2q,...,1$ then we have 
\begin{equation}\nonumber
 \tilde{m}_k=\rmin{\sqrt{-k^2-2qk}}{}\times\tan\left[\frac{\sqrt{-k^2-2qk}}{\rmin}\times\ln \r\right].
\end{equation}
Consequently,  we have for $k=-2q,...,-1$ 
\begin{equation}\nonumber
\begin{cases}\tilde m_k\geq (k^2+2qk)(1-\r)+\mathcal{O}[(1-\r)^2]\\\tilde m_0=0\end{cases}.
\end{equation}
Thus there is $1>\r_q^{(3)}\geq\r_q^{(2)}$ (depending on $q$) s.t. for $1>\r>\r_q^{(3)}$ we have for $k=-2q,...,-1$
\begin{equation}\label{Mk1Est}
\begin{cases}\tilde m_k\geq (k^2+2qk-10^{-6})(1-\r)\\\tilde m_0=0\end{cases}.
\end{equation}
{\bf Case III. $k>0$} \\

We now treat the last case: $k>0$. We study the minimization problems $\tilde{m}_k^{(l)}$ for $l=1,2,3$.

For $l=1,2,3$, we have [letting $\alpha= k^2+2qk$]
\[
\tilde{m}_k^{(l)}=\inf_{\substack{f\in H^1(]\r,1[,\C)\\\text{s.t.}f(1)=1}}\int_\r^1\left[r^{\beta_l+1}|f'|^2+r^{\beta_l-1}{\alpha}{}|f|^2\right]
\]
with 
\[
\beta_l=\begin{cases}2q&\text{if }l=1\\0&\text{if }l=2\\-2q&\text{if }l=3\end{cases}.
\]
By direct minimization, it is easy to see that $m_k^{(l)}$ admits a solution. Moreover a solution $f_l$ satisfies 
\[
\begin{cases}-(r^{\beta_l+1}f')'+{\alpha}r^{\beta_l-1}f=0&\text{for }r\in]\r,1[
\\f(1)=1\&f'(\r)=0
\end{cases}.
\]
From the ordinary differential equation we get that 
\[
f_l(r)=A_lr^{s_l}+B_lr^{t_l},\,A_l,B_l\in\C
\]
with 
\[
s _l=\frac{-\beta_l+\sqrt{\beta^2_l+4\alpha}}{2}\text{ and }t _l=\frac{-\beta_l-\sqrt{\beta^2_l+4\alpha}}{2}.
\]
Note that
\begin{equation}\label{ProdConj}
s _lt _l=-\alpha\text{ and }s_l-t_l=\sqrt{\beta^2_l+4\alpha}.
\end{equation}
For the simplicity of the presentation we drop the subscript $l$. 

\noindent From the boundary conditions we have
\[
\begin{cases}A+B=1\\As\r^s+Bt\r^t=0\end{cases}\Leftrightarrow\begin{cases}A=\dfrac{t\r^{t-s}}{t\r^{t-s}-s}\\B=\dfrac{s}{s-t\r^{t-s}}\end{cases}.
\]
As for the previous cases we have
\begin{eqnarray*}
\tilde{m}_k^{(l)}&=&f'_l(1)
\\&=&A_ls_l+B_lt_l
\\&=&\dfrac{s_lt_l\r^{t_l-s_l}}{t_l\r^{t_l-s_l}-s_l}+\dfrac{s_lt_l}{s_l-t_l\r^{t_l-s_l}}
\\&=&\dfrac{s_lt_l(1-\r^{t_l-s_l})}{s_l-t_l\r^{t_l-s_l}}
\\ {[\text{by \eqref{ProdConj}}]}&=&\dfrac{-\alpha(1-\r^{-\sqrt{\beta^2_l+4\alpha}})}{s_l-t_l\r^{-\sqrt{\beta^2_l+4\alpha}}}.
\end{eqnarray*}

In order to handle the expression of $\tilde{m}_k^{(l)}$, we note that for $\gamma\in\R$ we have $\r^\gamma=1-\gamma(1-\r)+\mathcal{O}[(1-\r)^2]$.

Therefore, for fixed $k\geq0$ we have [recall that $s_l-t_l=\sqrt{\beta^2_l+4\alpha}$]

\begin{eqnarray*}
\tilde m_k^{(l)}&=&\frac{-\alpha\left[1-\left(1+\sqrt{\beta^2_l+4\alpha}(1-\r)+\mathcal{O}[(1-\r)^2]\right)\right]}{s_l-t_l+t_l\sqrt{\beta^2_l+4\alpha}(1-\r)+\mathcal{O}[(1-\r)^2]}
\\
&=&\frac{\alpha\sqrt{\beta^2_l+4\alpha}(1-\r)+\mathcal{O}[(1-\r)^2]}{\sqrt{\beta^2_l+4\alpha}+t_l\sqrt{\beta^2_l+4\alpha}(1-\r)+\mathcal{O}[(1-\r)^2]}
\\
&=&\alpha(1-\r)+\mathcal{O}[(1-\r)^2].
\end{eqnarray*}
Consequently,  for $k\in\{1,...,2q\}$, we get
\[
\tilde m_k=(k^2+2qk)(1-\r)+\mathcal{O}[(1-\r)^2].
\]
Thus there is $1>\r_q^{(4)}\geq\r_q^{(3)}$ (depending on $q$) s.t. for $1>\r>\r_q^{(4)}$ and $k\in\{1,...,2q\}$ we have
\begin{equation}\label{Mk2Est}
\tilde m_k\geq (k^2+2qk-10^{-6})(1-\r)
\end{equation}
and 
\begin{equation}\label{Mk2EstBisa}
1-2q(1-R)>0.
\end{equation}
On the other hand, by noting  that $q^2+\alpha=(q+k)^2$ and that $q,k\geq0$, we have for fixed $R$ [when $k\to\infty$]
\begin{equation}\label{ExprTm1l}
\tilde{m}_k^{(1)}=\dfrac{(k^2+2qk)(1-\r^{2(q+k)})}{k\r^{2(q+k)}+2q+k}=(k+2q)(1+o_{k}(1)),
\end{equation}
\begin{equation}\label{ExprTm2l}
\tilde{m}_k^{(2)}=\dfrac{\sqrt{k^2+2qk}(1-\r^{2\sqrt{k^2+2qk}})}{1+\r^{2\sqrt{k^2+2qk}}}=(k+q)(1+o_k(1)),
\end{equation}
\begin{equation}\label{ExprTm3l}
\tilde{m}_k^{(3)}=\dfrac{(k^2+2qk)(1-\r^{2(q+k)})}{k+(2q+k)\r^{2(q+k)}}=(k+2q)(1+o_k(1)).
\end{equation}

From \eqref{ExprTm1l}, \eqref{ExprTm2l} and \eqref{ExprTm3l}, it is not difficult to prove that for $1>\r>\r_q^{(4)}$ there is $K_\r\geq2q+2$ (depending on $\r$ and $q$) s.t. for $k\geq K_\r$ we have that for $l=1,2,3$:
\begin{equation}\label{ExprTm4l}
\tilde{m}_k^{(l)}\geq k+\dfrac{1}{4}.
\end{equation}
Consequently from \eqref{ExprTm4l} we have for $k\geq K_\r$
\begin{eqnarray}\nonumber
\tilde{m}_k&=&\dfrac{1}{(1+\r^q)^2}\left[\tilde{m}_k^{(1)}+2\r^q\tilde{m}_k^{(2)}+\r^{2q}\tilde{m}_k^{(3)}\right]
\\\label{ExpreTilMk}&\geq&k+\dfrac{1}{4}.
\end{eqnarray}
And if $k\in\{2q+1,...,K_\r-1\}$ we just need
\begin{equation}\label{VeryEasyToGet}
\tilde{m}_k^{}>0.
\end{equation}
\subsection{Last computations and conclusion}
We are now in position to prove \eqref{NonExs6}.

On the one hand we have (with \eqref{DecSumTilL}, \eqref{Mk1Est} \eqref{Mk2Est}, \eqref{ExpreTilMk} and Proposition \ref{PropCoefV}.1)
\begin{eqnarray*}
&&\dfrac{\tilde{L}_\v(v_\v)}{\pi}-d\\&\geq&\sum_{k\in\Z}|a_k|^2(\tilde{m}_k-k)
\\&\geq&\sum_{k\leq-2q-1}|a_k|^2(\tilde{m}_k+|k|)+\sum_{k=-2q}^{-1}|a_k|^2\left[(k^2+2qk-10^{-6})(1-\r)+|k|\right]+\\
&&+\sum_{k=1}^{2q}|a_k|^2\left[(k^2+2qk-10^{-6})(1-\r)-k\right]+\sum_{k=2q+1}^{K_\r-1}|a_k|^2(\tilde{m}_k-k)+\\&&\phantom{tqfqfsfsqqqqqqqqqqqqqdddddddddddddfsfsfgshshjsjs}+\sum_{k\geq K_\r}\frac{|a_k|^2}{4}
\\&=&S_{1,2q}+S_{2q+1,K_\r-1}+S_{K_\r,\infty}.
\end{eqnarray*}
Where
\begin{eqnarray*}
S_{1,2q}&=&\sum_{k=1}^{2q}|a_k|^2\left[(k^2+2qk-10^{-6})(1-\r)-k\right]+\\&&\phantom{ahszhshhshshs}+|a_{-k}|^2\left[(k^2-2qk-10^{-6})(1-\r)+k\right],
\end{eqnarray*}
\begin{eqnarray*}
S_{2q+1,K_\r-1}&=&\sum_{k=2q+1}^{K_\r-1}k(|a_{-k}|^2-|a_k|^2)+|a_k|^2\tilde{m}_k+|a_{-k}|^2\tilde{m}_{-k},
\end{eqnarray*}
\begin{eqnarray*}
S_{K_\r,\infty}&=&\sum_{k\geq K_\r}\frac{|a_k|^2}{4}+|a_{-k}|^2(\tilde{m}_{-k}+k).
\end{eqnarray*}

From \eqref{Mk0Est} we have for $k\geq K_\r>2q$ that $\tilde{m}_{-k}>0$, then
\begin{equation}\label{FinalCCL1}
S_{K_\r,\infty}\geq\frac{1}{4}\sum_{k\geq K_\r}\{{|a_k|^2}+|a_{-k}|^2\}.
\end{equation}
By Proposition \ref{PropCoefV}.3, there are $C_1>0$ and $c_\v>0$ s.t. $c_\v\underset{\v\to\infty}{\to}1$ and for $k\in\{1,...,K_\r\}$ we have
\begin{eqnarray}\nonumber
|a_k|^2&\leq& c_\v^2|a_{-k}|^2+2c_\v|a_{-k}| C_1\sum_{l\in\Z^*}|a_l|^2+C_1^2\left(\sum_{l\in\Z^*}|a_l|^2\right)^2
\\\nonumber\text{[Proposition \ref{PropCoefV}.2]}&\leq&c_\v^2|a_{-k}|^2+o\left(\sum_{l\in\Z^*}|a_l|^2\right).
\end{eqnarray}
Consequently, for $k\in\{1,...,K_\r\}$ we have
\begin{eqnarray}\nonumber
|a_{-k}|^2-|a_{k}|^2&\geq&|a_{-k}|^2(1-c_\v^2)+o\left(\sum_{l\in\Z^*}|a_l|^2\right)
\\\label{FundComparBisrepet}\text{[$c_\v\to1\,\&\,$Proposition \ref{PropCoefV}.2]}&=&o\left(\sum_{l\in\Z^*}|a_l|^2\right)\text{ when }\v\to\infty.
\end{eqnarray}
We thus get
\begin{eqnarray}\nonumber
S_{1,2q}&=&\sum_{k=1}^{2q}\left\{|a_k|^2\left[(k^2+2qk-10^{-6})(1-\r)-k\right]\right.\\\nonumber&&\phantom{aaaaaaa}\left.+|a_{-k}|^2\left[(k^2-2qk-10^{-6})(1-\r)+k\right]\right\}
\\\nonumber&=&(1-R)\sum_{k=1}^{2q}(|a_k|^2+|a_{-k}|^2)(k^2-10^{-6})+\\\nonumber&&\phantom{aaaaccccc}+[1-2q(1-R)]\sum_{k=1}^{2q}\{|a_{-k}|^2-|a_{k}|^2\}
\\\stackrel{[\eqref{Mk2Est}, \eqref{Mk2EstBisa}\&\eqref{FundComparBisrepet}]}
{\geq}&&(1-R)\sum_{k=1}^{2q}(|a_k|^2+|a_{-k}|^2)(k^2-10^{-6})+o\left(\sum_{l\in\Z^*}|a_l|^2\right).\label{FinalCCL2} 
\end{eqnarray}
Clearly, from \eqref{Mk1Est}$\&$\eqref{VeryEasyToGet}, there is $\frac{1}{4}>\eta>0$ (independent of $\v$) s.t. we have 
\[
\begin{cases}\text{$\tilde{m}_k,\tilde{m}_{-k}>\eta$ for $k\in\{2q+1,...,K_\r-1\}$}\\(1-10^{-6})(1-R)>\eta\end{cases}
\]
and consequently (with \eqref{FundComparBisrepet})
\begin{equation}\label{FinalCCL3}
S_{2q+1,K_\r-1}\geq\eta\sum_{k=2q+1}^{K_\r-1}\{|a_k|^2+|a_{-k}|^2\}+o\left(\sum_{l\in\Z^*}|a_l|^2\right)\text{ when }\v\to\infty.
\end{equation}
Therefore, by combining \eqref{FinalCCL1}, \eqref{FinalCCL2} and \eqref{FinalCCL3} we have 
\begin{eqnarray*}
\dfrac{\tilde{L}_\v(v_\v)}{\pi}-d&\geq&S_{1,2q}+S_{2q+1,K_\r-1}+S_{K_\r,\infty}
\\&\geq&\eta\sum_{l\in\Z^*}|a_l|^2+o\left(\sum_{l\in\Z^*}|a_l|^2\right)
\\&>&0\text{ for sufficiently large $\v$.}
\end{eqnarray*}
This last result ends the proof of Theorem \ref{ThmNonExist}.

\section{Comments and perspectives}

In order to prove our results we have made several restrictions on the parameter $\e$, on the capacity of the domain and on the form of the domain (for Theorem \ref{ThmNonExist}). We want to discuss here why these restrictions appear and their necessity. \\

In Theorem \ref{PropEXist1} we assumed that the annular domain is "thin" (with large capacity) and that $\e$ is large. In view of Theorem \ref{nonexistMir} of Mironescu (see \cite{Size}) we know that if the annular domain is "thick" and if $\e$ is small then minimizers of $m_\e(p,p)$ do not exist (for $p\in \mathbb{N}^*$). However it is an open question to know if minimizers do exist for $\e$ large when the annular domain has small capacity for $p>1$. This is indeed the case for $p=1$, but for $p>1$ even for the Dirichlet energy $E_\infty$ this is not known. \\

In Theorem \ref{ThmNonExist} we also assumed that the annulus is "thin". The main reason for that is the following: in order to prove non existence of minimizers of $E_\e$ we want to show that for every $v\in \I_{p,q}$ 
$$E_\e(v) > m_\e(q,q)+ \pi(p-q)$$ if $p>q$.
However it is easier to compute the difference $E_\e(v)-m_\e(q,q)$ if the infimum $m_\e(q,q)$ is attained, since we can then use a decomposition Lemma (see \eqref{NonExs5}). For example when $m_\e(1,1)$ is not attained we know that $m_\e(1,1)=2\pi$ thanks to the Price Lemma \ref{PriceLemma}. Thus in order to prove non existence of minimizers in $\I_{p,1}$ for $p>1$ one could try to show that 
$$E_\e(v) >2\pi +\pi(p-1)$$ for all $v \in \I_{p,1}$. 

Other technical reasons appear in the process of the proof of Theorem \ref{ThmNonExist}. In  \cite{Misiats1} the author was able to get rid of the  technical restrictions on the size of the domain. Its argument does not apply in our case, this is mainly due to the fact that $|u_\e|$ does not converge to $1$ (or to a constant) when $\e \rightarrow +\infty$. 
The restriction on the shape of the domain in \ref{ThmNonExist} also comes from the fact that $|u_\e|$ does not converge to a constant as $\e \rightarrow +\infty$. More precisely we used in a crucial way that $\rho_\e^q > \rho_q=|u_\infty^q|$ in the proof of the Theorem. We also used that $\rho_\e^q$ only depends on $r$ in order to use a decomposition in Fourier series. We did not obtain analogous results in the case of a general annular domain. However we believe that Theorem \ref{ThmNonExist} holds for all annular type domain regardless of the shape or of the size.
\\

\noindent{\bf Acknowledgements.} The authors would like to thank Y. Ge, P. Mironescu and E. Sandier for fruitful discussions concerning this work.

\bibliographystyle{amsalpha}
\bibliography{Bibliosemistiff}
\end{document}